\documentclass[11pt]{article}

\title{Symplectic embeddings from concave toric domains into convex ones}
 
\author{Dan Cristofaro-Gardiner\footnote{Mathematics Department, University of California, Santa Cruz, Santa Cruz CA, USA.  \newline  Electronic address: dcristof@ucsc.edu.}}
\date{}

\usepackage{amssymb}
\usepackage{enumerate}
\usepackage{latexsym}
\usepackage{amsmath}
\usepackage{amsthm}
\usepackage{amscd}
\usepackage{mathabx}
\usepackage{graphicx}
\numberwithin{equation}{section}
\numberwithin{figure}{section}

\newtheorem{theorem}{Theorem}[section]
\newtheorem{proposition}[theorem]{Proposition}
\newtheorem{corollary}[theorem]{Corollary}

\newtheorem{lemma-definition}[theorem]{Lemma-Definition}

\theoremstyle{definition}
\newtheorem{definition}[theorem]{Definition}
\newtheorem{remark}[theorem]{Remark}

\newtheorem{example}[theorem]{Example}

\newcommand{\eqdef}{\;{:=}\;}

\newcommand{\C}{{\mathbb C}}

\newcommand{\R}{{\mathbb R}}

\newcommand{\op}{\operatorname}

\newcommand{\vu}{\nu}

\newcommand{\bpm}{\begin{pmatrix}}
\newcommand{\epm}{\end{pmatrix}}

\renewcommand{\epsilon}{\varepsilon}

\begin{document}

\setcounter{tocdepth}{2}

\maketitle

\begin{abstract}
Embedded contact homology gives a sequence of obstructions to four-dimensional symplectic embeddings, called ECH capacities.  In ``Symplectic embeddings into four-dimensional concave toric domains", the author, Choi, Frenkel, Hutchings and Ramos computed the ECH capacities of all ``concave toric domains", and showed that these give sharp obstructions in several interesting cases.   We show that these obstructions are sharp for all symplectic embeddings of concave toric domains into ``convex" ones.  In an appendix with Choi, we prove a new formula for the ECH capacities of convex toric domains, which shows that they are determined by the ECH capacities of a corresponding collection of balls.
\end{abstract}

\section{Introduction}

\subsection{The main theorem}

It is an interesting problem to determine when one symplectic manifold embeds into another.  In dimension 4, Hutchings' ``ECH capacities" \cite{qech} give one tool for studying this question.  ECH capacities are a certain sequence of nonnegative (possibly infinite) real numbers 
\[ 0 = c_0(X,\omega) \le \ldots \le c_k(X,\omega) \le \infty\] associated to any symplectic four-manifold $(X,\omega)$.  The key property they satisfy is the Monotonicity Axiom: if there exists a symplectic embedding
\[ (M_1,\omega_1) \to (M_2,\omega_2),\]
then we must have 
\begin{equation}
\label{eqn:monotonicityequation}
c_k(M_1) \le c_k(M_2)
\end{equation} for all $k$.  ECH capacities therefore give an obstruction to symplectically embedding one symplectic $4$-manifold into another.  

In \cite{concave}, the author, Choi, Frenkel, Hutchings, and Ramos used ECH capacities to study symplectic embeddings of ``toric domains".  A {\em toric domain} $X_{\Omega}$ is the preimage of a region $\Omega \subset \R^2$ in the first quadrant under the map
\[ \mu: \C^2 \to \R^2, \quad \quad (z_1,z_2) \to (\pi|z_1|^2,\pi|z_2|^2).\]
Toric domains generalize {\em ellipsoids}
\[ E(a,b) = \left\lbrace (z_1,z_2) | \frac{\pi|z_1|^2}{a} + \frac{\pi|z_2|^2}{b} \le 1 \right\rbrace,\]
 where $\Omega$ is a right triangle with legs on the axes, {\em balls} $B(c) \eqdef E(c,c)$, and {\em polydisks}
 \[ P(a,b) = \left\lbrace (z_1,z_2) | \frac{\pi|z_1|^2}{a} < 1,  \frac{\pi|z_2|^2}{b} \le 1\right\rbrace,\]
where $\Omega$ is a rectangle whose bottom and left sides are on the axes.  The paper \cite{concave} computed the ECH capacities of all ``concave" toric domains, and showed that these give sharp obstructions in several interesting cases, for example for all ball packings into certain unions of an ellipsoid and a cylinder.  The aim of the present article is to identify a large and natural class of embedding problems involving toric domains for which ECH capacities give a sharp obstruction.  It turns out that in these cases, ECH capacities can be computed purely combinatorially, and so give considerable insight into the corresponding embedding problem.       

\begin{figure}[t]
\label{fig:illustration}
\centering
\includegraphics[scale=.5]{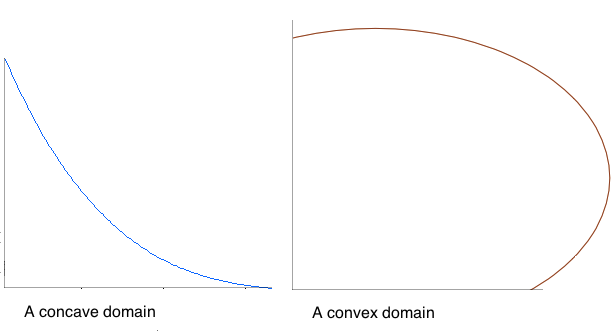}
\caption{A concave toric domain and a convex one}
\end{figure}

To state our main theorem, first recall the ``concave toric domains" from \cite{concave}.  These were defined as toric domains $X_{\Omega}$, where $\Omega$ is a region in the first quadrant underneath the graph of a convex function $f: [0,a] \to [0,b]$, such that $a$ and $b$ are positive real numbers, $f(0) = b$, and $f(a) = 0$. We call such an $\Omega$ a {\em concave} subset of the first quadrant of $\mathbb{R}^2$.    

We now define a related concept, see Figure~\ref{fig:illustration}.

\begin{definition} 
\label{defn:convextoric}
\label{def:convex} A {\em convex toric domain} is a toric domain $X_{\Omega}$, where $\Omega$ is a closed region in the first quadrant bounded by the axes and a convex curve from $(a,0)$ to $(0,b)$, for $a$ and $b$ positive real numbers. 
\end{definition}

Similarly to above, we call such an $\Omega$ a {\em convex} subset of the first quadrant.  Note that our definition of convex toric domain differs slightly from the definition in \cite{hutchings}.
 
If $X$ is a symplectic four-manifold, let $c_k(X,\omega)$ denote the $k^{th}$ ECH capacity of $X$, reviewed in \S\ref{sec:echreview}.  We can now state the main theorem of this paper:

\begin{theorem} 
\label{thm:maintheorem}
Let $X_{\Omega_1}$ be a concave toric domain and let $X_{\Omega_2}$ be a convex toric domain.  Then there exists a symplectic embedding
\[\op{int}(X_{\Omega_1}) \to \op{int}(X_{\Omega_2})\]
if and only if 
\[ c_k(\op{int}(X_{\Omega_1})) \le c_k(\op{int}(X_{\Omega_2}))\]
for all nonnegative integers $k$.   
\end{theorem}

We also remark that when an embedding of a concave toric domain into a convex one exists, it is unique up to isotopy, see Proposition~\ref{prop:connectivity}.  This is definitely not true for many other toric domains, see for example \cite{guttusher}.

Note that an ellipsoid is both concave and convex, while a polydisc is convex.  Thus, Theorem~\ref{thm:maintheorem} generalizes well-known results of McDuff \cite{m2} (where $X_{\Omega_1}$ and $X_{\Omega_2}$ are both ellipsoids) and Frenkel-M\"uller \cite{fm} (where $X_{\Omega_1}$ is an ellipsoid and $X_{\Omega_2}$ is a polydisc).  As mentioned above, a purely combinatorial formula for the ECH capacities of concave toric domains was given in  \cite{concave}.  In the appendix, we give a formula for the ECH capacities of convex domains that generalizes the formula from \cite[Thm. 1.11]{hutchings}, see Corollary~\ref{cor:convexcorollary}.  These formulas involve counting lattice points in polygons, and the combinatorics involved can be interesting \cite{cga, ms, cgfs}.

Here is an example of how one can use Theorem~\ref{thm:maintheorem}:

\begin{example}
\label{ex:mainexample}

Let $X_{\Omega_1}$ be an ellipsoid and let $X_{\Omega_2}$ be the convex toric domain associated to a closed symplectic toric four-manifold $X$.  This means that $\Omega_2$ is a Delzant polygon for $X$ (note that any Delzant polygon is affine equivalent to a polygon $\Omega_2$ which is convex in the sense of Definition~\ref{def:convex}).  Then $X$ contains the convex toric domain $X_{\Omega_2}$, so Theorem~\ref{thm:maintheorem} can be used to construct embeddings of ellipsoids into $X$.  In fact, it is shown in \cite{cghmp} that an ellipsoid embeds into $X$ if and only if it embeds into $X_{\Omega_2}$.  Thus, Theorem~\ref{thm:maintheorem} can be used to understand exactly when an ellipsoid embeds into a closed symplectic toric four-manifold.  This is studied in \cite{cghmp}, and the stabilized version of this embedding problem (as in \cite{cghi, cgm}) could be interesting to study as well. 
\end{example}

More examples are given in \S\ref{sec:examples}. 
    
\begin{remark}
\label{rmk:sharpness}

One could try to extend Theorem~\ref{thm:maintheorem} to other classes of toric domains.  However, it is important to note that ECH capacities definitely do not always give sharp obstructions to symplectic embeddings, even for toric domains.  A notable example of this is given by Hind and Lisi in \cite[Thm. 1.1]{hindlisi}, where it is shown that a polydisc $P(1,2)$ can be symplectically embedded into a ball $B(a)$ if and only if $a \ge 3$; ECH capacities only give the obstruction $a \ge 2.$  Interestingly, recent work of Hutchings \cite{hutchings} shows that embedded contact homology can still be used to derive strong obstructions to symplectic embeddings, even when the obstructions coming from ECH capacities are weak.  For example, in \cite{hutchings} Hutchings defines new obstructions to embedding one convex toric domain into another that can be used to recover the result of Hind and Lisi from above.  It is an interesting open question to determine how sharp these new obstructions are, for example for symplectic embeddings of one four-dimensional polydisc into another.  
\end{remark}

\subsection{Idea of the proof and relationship with previous work}
\label{sec:previous}

As mentioned above, McDuff showed that ECH capacities give a sharp obstruction to symplectically embedding one four-dimensional ellipsoid into another.  Here we use a similar method.  

Central to both methods is the {\em symplectic ball-packing problem}; for target a ball, this is the question of whether or not there exists a symplectic embedding
\[ \coprod_i B(a_i) \to B(\lambda)\]
for positive real numbers $\lambda, a_1, \ldots, a_n$.  McDuff showed in \cite{m1} that the question of whether or not one rational ellipsoid can be symplectically embedded into another is equivalent to the question of whether or not a certain symplectic ball packing of a ball exists.  In \cite{m4}, it was then shown that since ECH capacities are known to give a sharp obstruction to all four-dimensional symplectic ball packing problems of a ball, they give sharp obstructions to ellipsoid embeddings as well.   Here we first show that the question of  embedding a ``rational" concave toric domain into a rational convex one is equivalent to a certain symplectic ball packing problem, see Theorem~\ref{thm:keytheorem} for the precise statement, and we then use this to show that ECH capacities give a sharp obstruction to embedding a concave domain into a convex one.

\subsection{Connectivity of the space of embeddings}

McDuff also showed in \cite{m1} that the space of embeddings of one ellipsoid into another is connected.  To prove Theorem~\ref{thm:maintheorem} and Theorem~\ref{thm:keytheorem}, it will be helpful to show that this also holds for embeddings of a concave domain into a convex one:

\begin{proposition}
\label{prop:connectivity}
Let $X_{\Omega_1}$ be a concave toric domain, let $X_{\Omega_2}$ be a convex toric domain, and let $g_0$ and $g_1$ be two symplectic embeddings:
\[ X_{\Omega_1} \to \op{int}(X_{\Omega_2}).\]
Then there exists an isotopy \[\lbrace \Psi_t \rbrace_{0 \le t \le 1} : \op{int}(X_{\Omega_2}) \to \op{int}(X_{\Omega_2})\] 
such that $\Psi_0=\op{id}$ and $\Psi_1(g_1)=g_0$.  
\end{proposition}

The following corollary will be particularly useful:

\begin{corollary}
\label{cor:nonrational}
Let $X_{\Omega_1}$ be a concave domain and let $X_{\Omega_2}$ be convex.  Then there is a symplectic embedding
\[ \op{int}(X_{\Omega_1}) \to \op{int}(X_{\Omega_2})\]
if and only if   
there is a symplectic embedding
\[ X_{\lambda \Omega_1} \to \op{int}(X_{\Omega_2})\]
for all $\lambda < 1$.
\end{corollary}

\subsection{ECH capacities of convex domains and ECH capacities of balls}

As explained in \S\ref{sec:previous}, the fact that ECH capacities are sharp for these embedding problems essentially follows from the fact that they are sharp for symplectic ball packings of a ball.  In fact, the ECH capacities of both of these domains are closely related to the ECH capacities of balls.  In \cite{concave}, it was shown that the ECH capacities of any concave toric domain are determined by the ECH capacities of a certain collection of balls, see \cite[Thm. 1.4]{concave} for the precise statement.  In an appendix with Choi, we show that this is also true for convex toric domains, see Theorem~\ref{thm:echcapacities}.

\vspace{5 mm}

{\bf Acknowledgements.}  It is a pleasure to thank Dusa McDuff for her encouragement and her ideas, and for her help with this work.  I was motivated to think about Theorem~\ref{thm:maintheorem} because of a suggestion by McDuff.  I also thank Michael Hutchings and Emmanuel Opshtein for their encouragement.  It is a pleasure to thank Keon Choi and Felix Schlenk for valuable conversations as well.  I also thank the anonymous referee for valuable feedback.  
     
Part of the work for this project was completed while the author was visiting the Institute for Advanced Study.  I thank the Institute for their generous support, part of which came from NSF grant DMS-1128155.  I also benefited from discussions with McDuff and Hutchings at the Simons Center during the ``Moduli spaces of pseudoholomorphic curves and their applications to symplectic topology" program. I thank the center for their hospitality.  I also thank the NSF for their support under agreements DMS-1402200 and DMS-1711976.  

While this article was in its final stages, I learned that Opshtein has also recently studied symplectic embeddings of concave toric domains from a slightly different point of view, and found relationships with ball-packings, see \cite[\S 5]{opshtein}.   I thank Opshtein for his helpful correspondence about this.

\section{Weight sequences} 
\label{sec:weightsequences}

In \cite{m1}, McDuff introduced a set of real numbers determined by a $4$-dimensional symplectic ellipsoid, called a {\em weight sequence}.  We begin by explaining how to extend this construction to concave and convex toric domains.

\subsection{The concave case}

Weight sequences in the concave case were defined by Choi, the author, Frenkel, Hutchings, and Ramos in \cite{concave}.  We begin by reviewing this definition.

First, recall that two subsets of $\mathbb{R}^2$ are {\em affine equivalent} if one can be obtained from the other by multiplying by an element of $SL_2(\mathbb{Z})$ and applying a translation.  Now let $\Omega$ be a concave subset of the first quadrant of $\mathbb{R}^2$.  The weight sequence of $\Omega$ is an unordered set of (possibly repeated) nonnegative real numbers $w(\Omega)$ defined inductively as follows.  If $\Omega$ is a triangle with vertices $(0,0), (0,a)$ and $(a,0)$, then the weight sequence of $\Omega$ is $(a)$.  Otherwise, let $a > 0$ be the largest real number such that $\Omega$ contains the triangle with vertices $(0,0), (0,a)$ and $(a,0)$.  Call this triangle $\Omega_1$.  Then the line $x + y = a$ intersects the upper boundary of $\Omega$ in a line segment from $(x_1,a-x_1)$ to $(x_2,a-x_2)$, where $x_1 \le x_2$.  Let $\Omega'_2$ be the closure of the part of $\Omega$ to the left of $x_1$ and above this line, and let $\Omega'_3$ be the closure of the part of $\Omega$ to the right of $x_2$ and above this line, see Figure~\ref{fig:weightdef}.  Then, as explained in \cite[\S 1.3]{concave}, $\Omega'_2$ is affine equivalent to a canonical concave subset of the first quadrant, which we denote by $\Omega_2$.  Similarly, $\Omega'_3$ is affine equivalent to a canonical concave subset which will be denoted by $\Omega_3$.   We now define $w(\Omega)=w(\Omega_1) \cup w(\Omega_2) \cup w(\Omega_3)$, where $\cup$ denotes the (unordered) union with repetitions.  In the inductive definition, note that $w(\Omega)$ is defined to be $\emptyset$ if $\Omega=\emptyset$.  

If $X_{\Omega}$ is a concave toric domain, then we define the weight sequence of $X_{\Omega}$ to be $w(\Omega)$.

\begin{figure}[t]
\centering
\includegraphics[scale=.5]{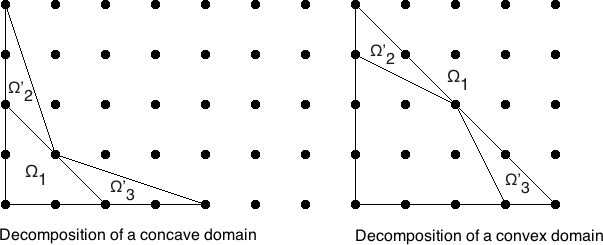}
\caption{The inductive decomposition of convex and concave toric domains}
\label{fig:weightdef}
\end{figure}

\subsection{The convex case}
\label{sec:convexblowup}

We now define a similar weight expansion for any convex toric domain.  The definition of the weight sequence for convex toric domains is similar to the definition of the weight sequence for concave toric domains.  If $\Omega$ is a triangle with vertices $(0,0), (0,b)$ and $(b,0)$ then the weight sequence of $\Omega$ is $(b)$.  Otherwise, let $b > 0$ be the smallest real number such that $\Omega$ is contained in the triangle with vertices $(0,0), (0,b)$ and $(b,0)$.  Call this triangle $\Omega_1$.  The line $x+y = b$ intersects the upper boundary of $\Omega$ in a line segment from $(x_1,b-x_1)$ to $(x_2,b-x_2),$ with $x_1 \le x_2$.  Let $\Omega_2'$ denote the closure of the portion of $\Omega_1 \setminus \Omega$ that is to the left of $x_1$ and below the line $x+y=b$, and let $\Omega_2'$ denote the closure of the portion of $\Omega_1 \setminus \Omega$ that is below $b-x_2$ and below the line $x + y = b$, see Figure~\ref{fig:weightdef}.  

The key point is now that $\Omega_2'$ and $\Omega_3'$ are both affine equivalent to concave subsets, which we denote by $\Omega_2$ and $\Omega_3$ respectively.  The equivalence for $\Omega_2'$ is given by translating down so that the top left corner of $\Omega_2'$ is at the origin, and then multiplying by the matrix $M = \left( \begin{smallmatrix} -1&-1\\ 1& 0 \end{smallmatrix} \right)$, while the equivalence for $\Omega_3'$ is given by translating so that the bottom right corner is at the origin, and then multiplying by the matrix $M'=\left(\begin{smallmatrix} 0 & 1 \\ -1 & -1 \end{smallmatrix}\right)$.  We then define
\[ w(\Omega) = (b; w(\Omega_2) \cup w(\Omega_3)).\]
Thus, the weight sequence for a convex set consists of a number, and then an unordered set of numbers.  We call the first number in this sequence the {\em head}, and we call the other numbers the {\em negative weight sequence}.  If $X_{\Omega}$ is a convex toric domain, then we define the weight sequence of $X_{\Omega}$ to be $w(\Omega)$.  

\subsection{Ball packings}

To simplify the notation, for a convex $\Omega,$ let 
\[ \widehat{B}(\Omega) = \coprod_i B(b_i),\]
where the $b_i$ are the negative weight expansion for $\Omega$.  Similarly, for a concave $\Omega$, let 
\[ B(\Omega)= \coprod_i B(a_i),\] 
where the $a_i$ are the weight expansion for $\Omega$.  Finally, call a concave or convex domain {\em rational} if it has upper boundary that is piecewise linear with rational slopes.  This guarantees that the weight sequence for this domain is finite.  

Here is the key result that we want to prove, in order to prove Theorem~\ref{thm:maintheorem}:

\begin{theorem}
\label{thm:keytheorem}
Let $X_{\Omega_1}$ be a rational concave toric domain, let $X_{\Omega_2}$ be a rational convex toric domain, and let $b$ be the head of the weight expansion for $\Omega_2$.  Then there exists a symplectic embedding
\[ \op{int}(X_{\Omega_1}) \to \op{int}(X_{\Omega_2})\]
if and only if there exists a symplectic embedding 
\[ \op{int}(B(\Omega_1)) \sqcup \op{int}(\widehat{B}( \Omega_2)) \to \op{int}(B(b)).\]
\end{theorem}

Note that the ``only if" direction of Theorem~\ref{thm:keytheorem} follows from the ``Traynor trick" \cite{traynor}, see e.g. \cite[Lem. 1.8]{concave} for the version we need, and the definition of the weight expansion.

\section{Embeddings of toric domains and embeddings of spheres}
\label{sec:ball}

We now begin the proof of Theorem~\ref{thm:keytheorem}.  We already showed the ``only if" direction, so we now show the converse.  In this section, we give the first part of the proof, which involves showing that  to embed a concave toric domain into a convex one, it is equivalent to embed a certain chain of spheres into a blowup of $\mathbb{C}P^2$.

\subsection{Preliminaries}
\label{sec:prelim}

We start by recalling those details of the symplectic blowup construction that are relevant to us.  Let $L$ denote the homology class of the line in $\mathbb{C}P^2$, and let $\omega_0$ denote the Fubini-Study form, normalized so that $\langle \omega_0,L \rangle = 1$.   Now suppose there is a symplectic embedding $\coprod_{i=1}^m B(a_i) \to (\mathbb{C}P^2,\omega_0)$.  We can remove the interiors of the $B(a_i)$ and collapse their boundaries under the Reeb flow to get a symplectic manifold, called the {\em blowup} of the ball packing, which is diffeomorphic to $\mathbb{C}P^2 \# m \overline{\mathbb{C}P^2}$, with a canonical symplectic form $\omega_1$.  The image of $\partial B(a_i)$ in this manifold is called the $i^{th}$ exceptional divisor.  If $E_i$ denotes the homology class of the $i^{th}$ exceptional divisor, then the cohomology class of $\omega_1$ is given by
\[ \op{PD}[\omega_1] = L - \sum_{i=1}^m a_i E_i.\]
Another class which will be relevant for our purposes is the {\em canonical class} $K$ defined by
\[ \op{PD}(K) \eqdef -3L + \sum_{i=1}^m E_i.\]
The class $K$ is $c_1(T^*X)$, as defined by any almost complex structure compatible with $\omega_1$.   
 
\subsection{Blowing up a concave domain}
\label{sec:blowup}

Now let $\Omega$ be any rational concave toric domain, and include $X_{\Omega}$ into some large ball $\op{int}(B(R))$, which we can include into a $(\mathbb{C}P^2,\omega)$.  We now mimic the definition of the weight sequence to define a sequence of symplectic blowups of $(\mathbb{C}P^2,\omega)$ that will produce one of the relevant chains of spheres, see Figure~\ref{fig:blowup} for an illustration.

Let $a$ be the largest real number such that $\Omega$ contains the triangle with vertices $(0,0), (0,a)$ and $(a,0)$, let $\delta>0$ be a sufficiently small real number, and consider the triangle $\Delta(a+\delta)$ with vertices $(0,0), (0,a+\delta)$ and $(a+\delta,0).$  Thus, in Figure~\ref{fig:blowup}, $\Delta(a+\delta)$ is the triangle with legs on both of the axes.  Then there is a symplectic embedding $B(a+\delta) \to B(R)$.  Blow up along $B(a+\delta).$  

Now the upper boundary of $\Delta(a+\delta)$ intersects the complement of $\Omega$ in the plane along a line segment between $(x_1,a+\delta-x_1)$ and $(x_2,a+\delta-x_2)$ with $x_1 < x_2$. Let $\Gamma_1$ be the closure of the subset of $\Omega$ which is to the left of $x_1$ and above the line $x + y = a + \delta$, and let $\Gamma_2$ be the closure of the subset of $\Omega$ which is to the right of $x_2$ and above this line.   Then, as in the definition of the weight sequence, $\Gamma_1$ and $\Gamma_2$ are affine equivalent to concave subsets.  

In the present context, this implies that we can iterate the procedure from the previous paragraph to perform a symplectic blowup for each element of the weight sequence for $\Omega$.   Each blowup produces a symplectic sphere.  The result of this sequence of blowups is a symplectic manifold $(\mathbb{C}P^2 \# m \overline{\mathbb{C}P^2},\omega_1)$ with a configuration of symplectic spheres $\mathcal{C}_{\Omega,\delta_{\Omega}}$, with one sphere for each element of the weight sequence.    Here, $\delta_{\Omega}$ denotes a sequence of small real numbers corresponding to the $\delta$ for each blow up.  

Later, we will to speak of blowing up $X_{\Omega}$ {\em with respect to an embedding} $g: X_{\Omega} \to M$.  This means performing the sequence of blowups from above in $M$, via the embedding $g$, and we will denote the resulting chain of spheres by  $\mathcal{C}_{g(\Omega),\delta_{\Omega}}.$

\begin{figure}[t]
\centering
\includegraphics[scale=.33]{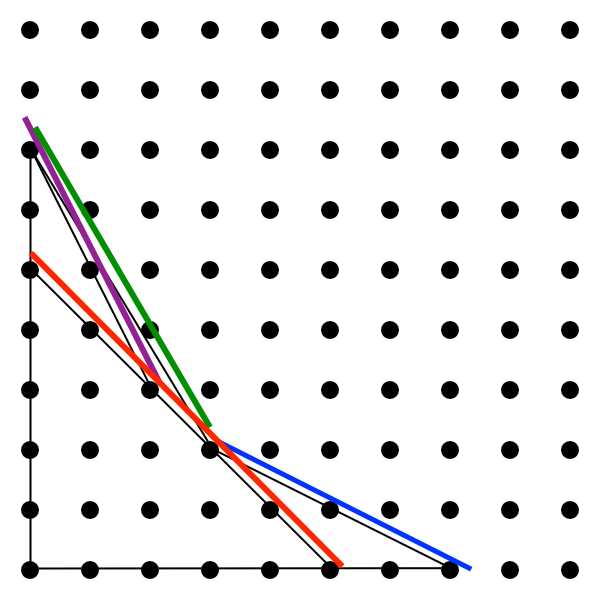}
\caption{Blowing up a rational concave domain $X_\Omega$.  In this case, the upper boundary of $\Omega$ consists of two line segments; to blow up $\Omega$, we perform four blowups, which are illustrated by the four thickened lines.  The first blow up removes the right triangle with legs on both axes, and upper boundary a thickened line.  The next two blow ups correspond to the two regions that touch one of the axes; the order in which we do these two blowups is irrelevant.  We then do one last blow up, corresponding to the triangle with all edges thickened lines.  The canonical weight sequence decomposition of the domain is also shown, in which we have partitioned $\Omega$ into four regions each of which are affine equivalent to right triangles.  The lines demarcating these regions are thin.}       
\label{fig:blowup}
\end{figure}
 
\subsection{Blowing up a convex domain} 
\label{sec:blowup2}

We now define a similar sequence of blowups if $\Omega$ is a rational convex domain.  Specifically, let $b$ be the head of the weight sequence for $\Omega$, and choose a small $\delta > 0$.  The line $x + y = b-\delta$ intersects $\Omega$ in a line segment from $(x_1,b-\delta-x_1)$ to $(x_2,b-\delta-x_2)$, where $x_1 < x_2$. Let $\Delta(b-\delta)$ be the triangle with vertices $(0,0), (b-\delta,0)$ and $(0,b-\delta)$.  Let $\Gamma_1$ be the closure of the region of the complement of $\Omega$ in $\Delta(b-\delta)$ that is to the left of $x_1$, and let $\Gamma_2$ be the closure of the region of the complement that is below $b-\delta-x_2$.  

We showed in the definition of the weight sequence that $\Gamma_1$ and $\Gamma_2$ are affine equivalent to concave toric domains.  We can therefore apply the procedure from \S\ref{sec:blowup} to associate a symplectic blow up of $(\mathbb{C}P^2, (b-\delta)\omega_0)$ to each term in the negative weight sequence for $\Omega$.  This gives a symplectic manifold $(\mathbb{C}P^2 \# n \overline{\mathbb{C}P^2},\omega_2)$ with a configuration of symplectic spheres which we denote by $\widehat{\mathcal{C}}_{\Omega,\delta_{\Omega}}$.  As in the previous section, $\delta_{\Omega}$ denotes a choice of small real numbers corresponding to the $\delta$ in this blow up construction.
  
\subsection{Inner and outer approximations}      

Our blowup procedure is closely related to the {\em inner and outer approximations} from \cite{m1}.  To elaborate, consider first the blow up procedure for rational concave $\Omega$.  Our blowup procedure shows that we can define another concave toric domain, called an {\em outer approximation} to $\Omega$, such that the sequence of blowups removes the interior of the outer approximation and collapses the boundary of the outer approximation to the configuration of spheres $\mathcal{C}_{\Omega,\delta}$. Denote the outer approximation to $\Omega$ by $\Omega^{out}_{\delta}$.   For example, in the situation illustrated in Figure~\ref{fig:blowup}, the outer approximation is the maximum concave set bounded by the axes and segments of the thickened lines. 

Similarly, if $\Omega$ is convex, then our blowup procedure shows that we can define another convex toric domain, called an {\em inner approximation} to $\Omega$, denoted $\Omega^{in}_{\delta},$ such that the sequence of blowups removes the complement of the inner approximation in $B(b-\delta)$ and collapses the boundary of the inner approximation to the configuration of spheres $\widehat{\mathcal{C}}_{\Omega,\delta_{\Omega}}$.  

\subsection{Embedding equivalences} 

The previous subsections defined chains of spheres $\mathcal{C}_{\Omega_1,\delta_{\Omega_1}} \sqcup \widehat{\mathcal{C}}_{\Omega_2,\delta_{\Omega_2}}$.  Define a {\em symplectic embedding of this chain} into a symplectic manifold $X$ to be a map of the disjoint union into $X$ such that the image of the spheres intersect transversally, the map restricts to each individual sphere as a symplectic embedding, and the intersection matrix of the chain in $X$ agrees with the intersection matrix of the chain $\mathcal{C}_{\Omega_1,\delta_{\Omega_1}} \sqcup \widehat{\mathcal{C}}_{\Omega_2,\delta_{\Omega_2}}$.  (Note that our intersection matrix includes the self-intersections of each sphere.)

\begin{proposition}
\label{prop:surgery}
Let $\Omega_1$ be a rational concave toric domain, and let $\Omega_2$ be a rational convex toric domain.  Let $m$ be the length of the weight expansion for $\Omega_1$, and let $n$ be the length of the negative weight expansion for $\Omega_2$.  If there is a symplectic form $\omega$ on $\mathbb{C}P^2 \# (m+n) \overline{\mathbb{C}P^2}$ such that there is a symplectic embedding 
\[\mathcal{C}_{\Omega_1,\delta_{\Omega_1}} \sqcup \widehat{\mathcal{C}}_{\Omega_2,\delta_{\Omega_2}} \to (\mathbb{C}P^2 \# (m+n) \overline{\mathbb{C}P^2},\omega),\]
then there is a symplectic embedding
\[ X_{\Omega_1} \to \op{int}(X_{\Omega_2}).\] 
\end{proposition}

To prove the proposition, we will need to use the following result: 

\begin{theorem} (Gromov-McDuff \cite[Thm. 9.4.2]{mcduff})
\label{thm:gromovmcduff}

Let $(M,\omega)$ be a connected symplectic $4$-manifold with no symplectically embedded $2$-spheres of self-intersection $-1$.  Assume that there exists a symplectomorphism 
\[ \Psi: \mathbb{R}^4 \setminus V \to M \setminus K,\]
where $K \subset M$ is compact, and $V \subset \mathbb{R}^4$ is compact and star-shaped with respect to the origin. Then for every open neighborhood $U$ of $K$, $(M,\omega)$ is symplectomorphic to $(\mathbb{R}^4,\omega_{std})$ 
by a symplectomorphism that agrees with $\Psi^{-1}$ on $M \setminus U$.

\end{theorem}

We can now give:

\begin{proof}[Proof of Proposition~\ref{prop:surgery}]
 
By assumption, there is a symplectic embedding
\[\mathcal{C}_{\Omega_1,\delta_{\Omega_1}} \sqcup \widehat{\mathcal{C}}_{\Omega_2,\delta_{\Omega_2}} \to (\mathbb{C}P^2 \# (m+n) \overline{\mathbb{C}P^2},\omega).\] 
We can make a small perturbation to this embedding so that all intersections are symplectically orthogonal, see for example \cite[Lem. 2.2]{m1}.

Now consider the embedding 
\[ \mathcal{C}_{\Omega_1,\delta_{\Omega_1}} \to (\mathbb{C}P^2 \# (m+n) \overline{\mathbb{C}P^2},\omega).\] 
A version of the symplectic neighborhood theorem \cite[Prop. 3.5]{symington} now implies that a neighborhood of these spheres can be identified with a neighborhood of the chain of spheres in the manifold $(\mathbb{C}P^2 \# m \overline{\mathbb{C}P^2},\omega_1)$ that was constructed in \S\ref{sec:blowup} by blowing up the outer approximation.  We can therefore remove the $\mathcal{C}_{\Omega_1,\delta_{\Omega_1}}$ and glue in a copy of $X_{\Omega^{out}_{1,\delta_1}}$ to get a new symplectic manifold $\tilde{Z}$ which admits a symplectic embedding of $X_{\Omega_1}$ whose image avoids a small neighborhood of $\widehat{\mathcal{C}}_{\Omega_2,\delta_{\Omega_2}}$.  By again applying the symplectic neighborhood theorem \cite[Prop. 3.5]{symington}, this neighborhood can be identified with a neighborhood of the chain of spheres constructed in \S\ref{sec:blowup2}  by blowing up the inner approximation.

Let $Z$ denote the complement of $\widehat{\mathcal{C}}_{\Omega_2,\delta_{\Omega_2}}$ in $\tilde{Z}$.  As above, we can glue in a copy of $\mathbb{R}^4 \setminus X_{\op{int}(\Omega^{in}_{2,\delta_{\Omega_2}})}$ to $Z$.  This gives a symplectic manifold $(M,\omega)$.  Since $H_2(M)=0$, $M$ can not contain any symplectically embedded $-1$ spheres.  For $r<1$ close to $1$, we can choose a symplectomorphism $\Psi: M \setminus K \to \mathbb{R}^4 \setminus X_{r \cdot \Omega^{in}_{2,\delta_{\Omega_2}}}$ for some compact $K \subset M$.  The set $X_{r \cdot \Omega^{in}_{2,\delta_{\Omega_2}}}$ is star-shaped with respect to the origin, since  if $v \in X_{r \cdot \Omega^{in}_{2,\delta_{\Omega_2}}}$ then so is $tv$ for all $0 \le t \le 1$ as $r \cdot \Omega^{in}_{2,\delta_{\Omega_2}}$ is itself star-shaped with respect to the origin.   Now regard $Z$ as a neighborhood of $K$ in $M$, and apply Theorem~\ref{thm:gromovmcduff} to $M$ with $U = Z$.  This produces a symplectomorphism between $(M,\omega)$ and $(\mathbb{R}^4,\omega_{std})$ that maps $Z$ to $X_{\op{int}(\Omega^{in}_{2,\delta_{\Omega_2}})}$.
 
\end{proof}

\section{Applying inflation}

In this section, we prove Theorem~\ref{thm:keytheorem}. 

Let $\Omega_1$ be concave and $\Omega_2$ convex.  We already proved the ``only if" direction of the theorem, so we just have to produce an embedding $\op{int}(X_{\Omega_1}) \to \op{int}(X_{\Omega_2})$, assuming that a certain ball packing exists.    By Proposition~\ref{prop:surgery}, it suffices to find a symplectic embedding 
\[\mathcal{C}_{\Omega_1,\delta_{\Omega_1}} \sqcup \widehat{\mathcal{C}}_{\Omega_2,\delta_{\Omega_2}} \to (\mathbb{C}P^2 \# (m+n) \overline{\mathbb{C}P^2},\omega),\]
where $\omega$ is any symplectic form.  If we choose $r$ sufficiently small, then we can construct a symplectic embedding  
\[\mathcal{C}_{r \cdot \Omega_1,\delta_{\Omega_1}} \sqcup \widehat{\mathcal{C}}_{\Omega_2,\delta_{\Omega_2}} \to (\mathbb{C}P^2 \# (m+n) \overline{\mathbb{C}P^2},\omega)\]
for some $\omega$, by the procedure in \S\ref{sec:blowup} and \S\ref{sec:blowup2}.  We now want to change the areas of the spheres in $\mathcal{C}_{r \cdot \Omega_1,\delta_{\Omega_1}}$, keeping them symplectic.  We accomplish this by using the ``inflation" method, from (for example) \cite{m1,mo,bi}.

\subsection{Review of inflation}
\label{sec:inflation}

We begin by reviewing the inflation method.

We first need to recall the aspects of Taubes' ``Seiberg-Witten = Gromov" theorem that we will need.  Let $(X,\omega)$ be a closed symplectic $4$-manifold,  and let $b_2^+(X)$ denote the dimension of the maximal subspace $H_2^+(X,\mathbb{R})$ of $H^2(X,\mathbb{R})$ on which the intersection form is positive definite.  If $A \in H_2(X;\mathbb{Z})$ and $b_2^+(X) \ge 1,$ then {\em Taubes' Gromov invariant}   $Gr(A)$ is defined by counting certain mostly embedded $J$-holomorphic curves in class $A$, for generic $\omega$-compatible almost complex structure $J$, see for example  \cite[\S 2]{hutchingslecture} for details. A symplectic form $\omega$ defines a $spin^c$ structure $\mathbf{s}_{\omega}$, and the ``Seiberg-Witten = Gromov" theorem \cite{taubesgr} states that there is an equivalence 
\begin{equation}
\label{eqn:swgr}
Gr(A) = SW(A),
\end{equation}
where $SW(A)$ denotes the {\em Seiberg-Witten invariant} of the $spin^c$ structure $\mathbf{s}_{\omega} + \op{PD}(A),$ see \cite{kmbook}.   When $b_2^+(X) = 1$, which is the situation in the present work, the Seiberg-Witten invariant also depends on a choice of ``chamber", which we can identify with a choice of orientation of the line $H_2^+(X;\mathbb{R})$; in this case, in \eqref{eqn:swgr} we choose the chamber determined by the cohomology class of the symplectic form.   

Now recall that a {\em symplectic divisor} is a union of symplectically embedded surfaces which intersect transversally and symplectically orthogonally, while an {\em exceptional class}  $\mathcal{E}  \in H_2(X)$ is a class $\mathcal{E}$ represented by a symplectically embedded $-1$ sphere. Here is the main result from the theory of inflation that we use:

\begin{proposition}\cite[Lem. 1.2.11]{mo}
\label{thm:moinflation}
Let $(X,\omega)$ be a symplectic manifold, $A \in H_2(X)$, and $S$ a symplectic divisor.  Assume: 
\begin{enumerate}[(i)]
\item $A \cdot A > 0,$ \label{thiswork}
\item $A \cdot \mathcal{E} \ge 0$ for all exceptional classes $\mathcal{E}$, \label{thiswork2}
\item $A \cdot S_i \ge 0$ for every component $S_i$ of $S$, \label{thiswork3} 
\item $X$ has nonzero Gromov invariant in class $A$.
\end{enumerate}
Then for any $s \ge 0$, the class
\[  [\omega] + s\op{PD}(A)\]
has a symplectic representative that is nondegenerate on $S$.
\end{proposition}

The idea of the proof of Proposition~\ref{thm:moinflation} is that since the manifold $X$ has nontrivial Gromov invariant in class $A$, we can find a symplectic submanifold $T$ in class $A$.  We use condition (\ref{thiswork2}) to guarantee that $T$ is connected, and then (\ref{thiswork}) to guarantee that $T$ has positive self intersection.  We then deform $\omega$ locally around $T$ in the normal direction by adding a certain closed $2$-form that is $0$ along $T$, see also Remark~\ref{rmk:illustration} below.  The condition (\ref{thiswork3}) is needed to guarantee that these deformations of $\omega$ remain symplectic along $S$.   A significant complication occurs when $S$ has at least one component that can not be made $J$-holomorphic and transverse for any $J$ because of index considerations, because one can not simultaneously guarantee that $S$ is $J$-holomorphic while $J$ is also suitably generic for defining Taubes' Gromov invariant; however, these difficulties can be overcome.  For more details, see \cite{mo}.

\begin{remark}
\label{rmk:illustration}

The simpler case where $A \cdot A = 0,$ $\mathcal{S}$ is empty, and $A$ has an embedded connected symplectic representative $T$ is illustrative.  Since $A \cdot A = 0$, the normal bundle of $C$ is trivial, and a neighborhood of $C$ can be identified symplectically with the symplectic product
\begin{equation}
\label{eqn:localdecomp}
(C \times D^2, \omega|_C \times \omega_{std}),
\end{equation}
where $D^2$ is a small disc.  To find a deformation through symplectic forms as in Proposition~\ref{thm:moinflation}, we locally add $(0, g(r) \omega_{std})$ to $\omega$ in the neighborhood given by \eqref{eqn:localdecomp}, where $g(r)$ is a nonnegative 
bump function.

In general, we can identify a neighborhood of $T$ with a neighborhood of the zero section in a complex Hermitian line bundle $\pi: E \to C$ of degree $A \cdot A$, so that the symplectic form is given by 
\[ \pi^*(\omega|_C) + d(\pi r^2 \beta),\]
where $r$ is the radial distance function, and $\beta$ is a certain connection $1$-form on the unit circle bundle.   We can now add 
\[ -d( g(r) \beta) \]
to the symplectic form, where $g(r)$ is an appropriately chosen bump function, see eg \cite[Lem. 1.1]{m3} for the details.  The requirement $A \cdot A \ge 0$ is required to ensure that the form remains symplectic for large $s$; in Proposition~\ref{thm:moinflation}, we demand in addition that $A \cdot A > 0$ to avoid potential complications coming from multiply covered torii in Taubes' Gromov invariant, although this assumption could be weakened.
\end{remark}

To apply Proposition~\ref{thm:moinflation}, we need conditions guaranteeing that certain classes have nonzero Gromov invariant.  By \eqref{eqn:swgr}, it is equivalent to find spin$^c$ structures with nonvanishing Seiberg-Witten invariant.  The Seiberg-Witten invariants of blow ups of $\mathbb{C}P^2$ were studied by Kronheimer and Mrowka in \cite{km}.  Their results, combined with \eqref{eqn:swgr} and known properties of the Seiberg-Witten invariant give the following, see Remark~\ref{rmk:howdoesthiswork} below:

\begin{proposition} 
\label{prop:nonzero}
Let $(M,\omega)$ be a symplectic blow up of $\mathbb{C}P^2$, and let $A \in H_2(M; \mathbb{Z}).$  Assume that
\begin{equation}
\label{eqn:coolequations}
A^2 - K \cdot A \ge 0, \quad \quad [\omega] \cdot ( \op{PD}(K) - A ) < 0.
\end{equation}
Then $Gr(A) \ne 0$.
\end{proposition}

\begin{remark}
\label{rmk:howdoesthiswork}

A sketch of the proof of Proposition~\ref{prop:nonzero} is valuable in order to understand why the conditions \eqref{eqn:coolequations} appear.  The equation
\begin{equation}
\label{eqn:negativeareaequation}
 [\omega] \cdot (\op{PD}(K) - A) < 0
\end{equation}
guarantees that the manifold $M$ has vanishing Gromov invariant in the class $\op{PD}(K) - A$, since \eqref{eqn:negativeareaequation} implies that any curve in class $\op{PD}(K) - A$ would have to have negative area.  Now the Seiberg-Witten invariants satisfy a basic symmetry, called {\em charge conjugation}.  In the case where $b_2^+ \ge 2,$ charge conjugation states that the Seiberg-Witten invariants of a spin$^c$ structure and its dual structure are the same up to sign, which one expects from examining the unperturbed equations and applying complex conjugation.  In the situation of Proposition~\ref{prop:nonzero}, where $b_2^+ = 1$, a similar fact holds except that there is an additional complication coming from the choice of chamber; the upshot for our purposes is that when we combine the charge conjugation relation in the case $b_2^+ = 1$ with \eqref{eqn:swgr}, we find that
\begin{equation}
\label{eqn:conjugationrelation}
Gr(A) - Gr(\op{PD}(K - A)) =  w(A) \quad \quad \op{mod} \hspace{1 mm} 2,
\end{equation}
where $w(A)$ is the {\em wall-crossing number} which is defined as the difference between the two chambers of the Seiberg-Witten invariants of the spin$^c$ structure corresponding to $A$, counted modulo\footnote{There is a version of \eqref{eqn:conjugationrelation} that holds without reducing modulo $2$ but we do not need this.}  $2$. This wall-crossing number was computed in the cases we need by Kronheimer-Mrowka: as explained in \cite[Thm. 9.9]{salamon}, the condition
\[ A^2 - K \cdot A \ge 0 \]
implies that 
\begin{equation}
\label{eqn:wallcrossing}
w(A)=1.
\end{equation} 
Since $Gr(\op{PD}(K) - A) = 0$ as explained above, combining \eqref{eqn:conjugationrelation} with \eqref{eqn:wallcrossing} implies that $Gr(A) \ne 0$. 

\end{remark}

We will also need a ``family" version of Proposition~\ref{thm:moinflation}.  To state the variant that we use, recall that two symplectic forms are called {\em deformation equivalent} relative to a symplectic divisor $S$ if there is a family of symplectic forms between them that restrict to nondegenerate forms on $S$; they are called {\em isotopic} relative to $S$ if one can choose this family to have constant cohomology class.  We call such a family a {\em connecting isotopy}.

\begin{theorem} \cite[Thm. 1.2.12]{mo}
\label{thm:familyinflation}
Let $(M,\omega)$ be a symplectic blow up of $\mathbb{C}P^2$ and let $\omega'$ be any symplectic form in the same cohomology class as $\omega$.  Assume that $\omega$ and $\omega'$ are deformation equivalent relative to $S$.  Then $\omega$ and $\omega'$ are isotopic relative to $S$.  Moreover, if $\omega = \omega'$ near $S$ then we can choose the connecting isotopy to be constant near $S$.
\end{theorem}

The assumption in Theorem~\ref{thm:familyinflation} that $(M,\omega)$ is a blowup of $\mathbb{C}P^2$ is sufficient for our purposes, but can be weakened; probably all that is needed is that $b_2^+(X) = 1$ so that $X$ has enough nonvanishing Seiberg-Witten invariants, see \cite[Rmk. 1.2.14]{mo}.

\subsection{Connectivity}

Having reviewed the inflation method, we can now give the proof of Proposition~\ref{prop:connectivity}, which states that the space of embeddings from a concave domain into a convex one is connected.  We also prove Corollary~\ref{cor:nonrational}.

\begin{proof}[Proof of Proposition~\ref{prop:connectivity}]

The proof closely follows the proof of \cite[Cor. 1.6]{m1}.  

First, assume that $\Omega_1$ and $\Omega_2$ are rational, and let $g_0$ and $g_1$ be symplectic embeddings of $X_{\Omega_1}$ into $\op{int}(X_{\Omega_2})$.  By applying Alexander's trick, see e.g. the proof of \cite[Prop. A.1]{schlenk}, we can assume that $g_0$ and $g_1$ agree with the inclusion of $X_{r \Omega_1}$ into $\op{int}(X_{\Omega_2})$ for sufficiently small $r$.  

We will produce an isotopy between $g_0$ and $g_1$ by using Theorem~\ref{thm:familyinflation}.  Namely, as explained in \S\ref{sec:blowup}, we can blow up $X_{\Omega_1}$ with respect to $g_0$ to get a symplectic manifold $(X_0,\omega_0)$ with a symplectic divisor $S = \mathcal{C}_{g_0(\Omega_1), \delta_{ \Omega_1}} \sqcup \mathcal{\widehat{C}}_{\Omega_2,\delta_{\Omega_2}}.$  We can produce a family of symplectic forms on $X_0$ starting at $\omega_0$ by first blowing up $t \cdot X_{\Omega_1}$ with respect to $g_0$ as $t$ ranges from $1$ to $r$, and then blowing up $t \cdot X_{\Omega_1}$ with respect to $g_1$ as $t$ ranges from $r$ to $1$, while identifying the underlying smooth manifolds of these blow ups with $X_0$ as in Step $2$ of \cite[\S 3]{m3}.  This implies in particular that the symplectic form $\omega'$ on $X_0$, given by blowing up along $g_1$, is deformation equivalent to the symplectic form $\omega=\omega_0$.  We can assume in addition that $\omega = \omega'$ near $S$.  

Now apply Theorem~\ref{thm:familyinflation}.  This gives an isotopy of symplectic forms on $X_0$ that is constant near $S$.  By Moser's trick, this gives an isotopy $\hat{\Psi}_t$ of the symplectic manifold $X_0$, which we can blow down to get an isotopy $\Psi_t$ of $\Omega_2$ taking $g_0$ to $g_1$.

{\em Step 2.}  We now deduce the general case from the rational one.  

We can extend the embeddings $g_0$ and $g_1$ to an open neighborhood of $\Omega_1,$ and so we can find a rational concave set $\Omega'_1$ satisfying
\[ \Omega_1 \subset \Omega'_1 \]
such that $g_0$ and $g_1$ give symplectic embeddings
\[ X_{\Omega'_1} \to \op{int}(X_{\Omega_2}).\]
We can then pick a rational convex set $\Omega'_2$ with
\[ \Omega'_2 \subset \Omega_2\]
so that the images of $X_{\Omega'_1}$ under $g_0$ and $g_1$ lie in $\op{int}(X_{\Omega'_2}).$  By the previous step, we can find an isotopy of $\op{int}(X_{\Omega'_2})$ taking $g_0$ to $g_1$.  Moreover, since the isotopy of symplectic forms in the previous step was constant near $S$, we can extend this isotopy to $\op{int}(X_{\Omega_2}).$

\end{proof}

\begin{proof}[Proof of Corollary~\ref{cor:nonrational}]

Since for $\lambda < 1,$ $X_{\lambda \Omega_1} \subset X_{\Omega_1}$, an embedding
\[ \op{int}(X_{\Omega_1}) \to \op{int}(X_{\Omega_2})\]
induces an embedding
\[ \op{int}(X_{\lambda \Omega_1}) \to \op{int}(X_{\Omega_2})\]
by composition.

In the other direction, given symplectic embeddings
\[X_{\lambda \Omega_1} \to \op{int}(X_{\Omega_2})\]
for all $\lambda <1$, we can choose a sequence of embeddings
\[g_n:X_{(1-1/n) \Omega_1} \to \op{int}(X_{\Omega_2}).\]
By applying Proposition~\ref{prop:connectivity}, we can further assume that this sequence of maps is nested.  We can therefore construct the desired symplectic embedding by taking the direct limit.

\end{proof}

\subsection{Inflating the spheres}

We can now complete the proof of Theorem~\ref{thm:keytheorem} by using the inflation procedure.

\begin{proof}[Proof of Theorem~\ref{thm:keytheorem}]

Let $\Omega_1$ be concave and $\Omega_2$ convex.  We already showed the ``only if" direction of the theorem, so we just have to prove the converse.

{\em Step 1.}  By assumption, there is a symplectic embedding
\begin{equation}
\label{eqn:givenequation}
\op{int}(B(\Omega_1)) \sqcup \op{int}(\hat{B}(\Omega_2)) \to \op{int}(B(b))).
\end{equation}
Let the $a_i$ be the weights for $\Omega_1$ and the $b_i$ the negative weights for $\Omega_2$, as defined in \S\ref{sec:weightsequences}.    

Because of the existence of the embedding \eqref{eqn:givenequation}, we can find a symplectic embedding 
\begin{equation}
\label{eqn:ballpacking}
(\coprod_i B(a_i')) \sqcup (\coprod_j B(b_j')) \to \op{int}(B(b)),
\end{equation}   
where the $a_i'$ and the $b_i'$ are strictly smaller than the corresponding $a_i$ and $b_i$, but otherwise as close as we wish.  For any $\lambda > 1$, we can in addition choose the $a_i', b_i'$ so that $(b; b_1', \ldots, b_n')$ is the weight sequence for a rational convex toric domain $\Omega'_2$ with the property that 
\begin{equation}
\label{eqn:firstlambdaequation}
\Omega'_2 \subset \lambda \cdot \Omega_2,
\end{equation}
while the $a_i'$ are the weights for a rational concave toric domain $\Omega'_1$ with
\begin{equation}
\label{eqn:secondlambdaequation}
\frac{1}{\lambda} \cdot \Omega_1 \subset \op{int}(\Omega'_1).
\end{equation} 
We can also assume that $b$, the $a_i'$ and the $b_i'$ are all rational.  

We will show that because of the existence of an embedding \eqref{eqn:ballpacking}, there is an embedding
\begin{equation}
\label{eqn:neededembedding}
\op{int}(X_{\Omega'_1}) \to X_{\Omega'_2}.
\end{equation}

{\em Step 2.}  Let $r$ be small enough that $r \cdot \Omega_1' \subset \op{int}(\Omega_2').$   Since $r \cdot  \Omega_1'$ is a concave toric domain, and $\Omega_2'$ is a convex toric domain,  we can apply the iterated blowup procedure from \S\ref{sec:blowup} and \S\ref{sec:blowup2} to conclude that there is a symplectic embedding
\[S = \mathcal{C}_{r \cdot \Omega_1', \delta_{r \cdot \Omega_1'}} \sqcup \mathcal{\widehat{C}}_{\Omega'_2,\delta_{\Omega'_2}} \to (\mathbb{C}P^2 \# (m+n) \overline{\mathbb{C}P^2},\omega_1).\]  

Let $L$ denote the homology class of the line in this blowup, let $E_1,\ldots,E_m$ be the exceptional classes associated to the blow ups for $r \cdot \Omega'_1$, and let $\widehat{E}_1,\ldots,\widehat{E}_n$ be the exceptional classes associated to the blow ups for $\Omega'_2$.  Let $\ell = \op{PD}(L)$, let $e_i=\op{PD}(E_i)$, and let $\widehat{e}_j=\op{PD}(\widehat{E}_j)$.  By \S\ref{sec:prelim} we know that the cohomology class of $\omega_1$ is given by
\[ [\omega_1] = b \ell - \sum_{i=1}^m (r \cdot a'_i) e_i - \sum_{j=1}^n b_j' \hat{e}_j - err(\delta),\]
where $\op{err}(\delta)$ denotes the error term coming from the $\delta_i$ parameters in the iterated blowup construction, and limits to $0$ as the $\delta_i$ do.

To show that a symplectic embedding \eqref{eqn:neededembedding} exists, we will show that there is a symplectic embedding 
\[\mathcal{C}_{\Omega'_1, \delta_{\Omega'_1}}\sqcup \mathcal{C}_{\Omega'_2,\delta_{\Omega'_2}} \to (\mathbb{C}P^2 \# (m+n) \overline{\mathbb{C}P^2},\omega),\]
for some symplectic form $\omega$, so that we can appeal to Proposition~\ref{prop:surgery}.  The intersection matrix for the configuration $\mathcal{C}_{\Omega'_1, \delta_{\Omega'_1}}\sqcup \mathcal{C}_{\Omega'_2,\delta_{\Omega'_2}}$ is the same as the intersection matrix for $S$.  Our strategy is then to find a symplectic form $\omega_2$, different from $\omega_1$, that restricts to $S$ as a nondegenerate form with the property that the spheres in $S$ have the same areas as the spheres in $\mathcal{C}_{\Omega'_1, \delta_{\Omega'_1}}\sqcup \mathcal{C}_{\Omega'_2,\delta_{\Omega'_2}}$.    

{\em Step 3.}  Consider the rational homology class 
\[ A \eqdef bL - \sum_{i=1}^m a_i' E_i - \sum_{j=1}^n b_j' \hat{E}_j,\]
and choose a positive integer $k$ such that $kA$ is integral.  We want to apply Proposition~\ref{thm:moinflation} to $kA$ on the manifold $M=(\mathbb{C}P^2 \# (m+n) \overline{\mathbb{C}P^2},\omega_1)$, for sufficiently large integer $k$.  To do this, we have to check that the four conditions in the assumptions of Proposition~\ref{thm:moinflation} are satisfied.

The conditions in \eqref{eqn:coolequations} are satisfied for sufficiently large $k$, so by Proposition~\ref{prop:nonzero} we can assume that the manifold $M$ has nonzero Gromov invariant in class $kA$.  The condition $k^2 (A \cdot A) > 0$ holds because of the existence of the embedding \eqref{eqn:ballpacking}, since symplectic embeddings have to preserve volume.  

To see that the second condition in Proposition~\ref{prop:nonzero} holds, let $\mathcal{E}$ be an exceptional class.  Then by a result of Li-Li \cite{lili}, $E$ is an exceptional class for any symplectic form on $\mathbb{C}P^2 \# (m+n) \overline{\mathbb{C}P^2}$.  
In particular, the symplectic form $\omega$ that comes from blowing up along the ball packing \eqref{eqn:ballpacking} has positive pairing with $\mathcal{E}$, which implies that $k \cdot A$ does as well.    

Finally, we can see that $k A $ intersects any sphere in $\mathcal{S}$ nonnegatively as follows.  First, let $S_i$ be an element of $\mathcal{\widehat{C}}_{\Omega'_2,\delta_{\Omega'_2}}.$  Then the homology class $[S_i]$ of $S_i$ is in the span of $L, \hat{E}_1, \ldots, \hat{E}_n$, so
\[ A \cdot [S_i] = (bL - \sum_{j=1}^n b'_j \hat{E}_j) \cdot [S_i].\]  
Since $[\omega_1] \cdot S_i > 0$, we have
\begin{equation}
\label{eqn:positivityequation}
(bL- \sum^n_{j=1} b'_j \hat{E}_j) \cdot [S_i] - err(\delta) \cdot [S_i] > 0.
\end{equation}
We can choose the $\delta_i$ in the blow up construction as small as we would like, and the relation \eqref{eqn:positivityequation} remains true.  Hence, since $err(\delta)$ goes to $0$ as the $\delta_i$ do, we have $A \cdot [S_i] \ge 0$ by continuity.

The case where $S_i$ is an element of $\mathcal{C}_{\Omega'_1, \delta_{\Omega'_1}}$ is analogous.  In this case, $[S_i]$ is in the span of $E_1, \ldots, E_m.$  Hence
\[ A \cdot [S_i] = (- \sum_{i=1}^m a'_i E_i) \cdot [S_i].\]   Now choose $R$ such that $\Omega'_1 \subset R \cdot \op{int}(\Omega'_2)$, and blow up to get a symplectic form  on $\mathbb{C}P^2 \# (m+n) \overline{\mathbb{C}P^2}.$   This form pairs positively with $[S_i]$, so we can repeat the argument from the previous paragraph to conclude that $A \cdot [S_i] \ge 0$ in this case as well.

{\em Step 4.}   By the previous step, we are now justified in applying Proposition~\ref{thm:moinflation}.  For any $s \ge 0$, this gives a symplectic form $\omega_{1,s}$ in cohomology class
\[ [\omega_{1,s}] = [\omega_1] + s k \cdot \op{PD}[A]\]
that is nondegenerate along $S$.  Now consider the symplectic form $\frac{1}{1+sk} \omega_{1,s}.$  We have 
\[ \frac{1}{1+sk} [\omega_{1,s}] = b \ell - \sum_{j=1}^n b'_j \hat{e}_j - \sum_{i=1}^m a'_i  \left(\frac{r+sk}{1+sk}\right) e_i - \frac{1}{1+sk} err(\delta). \]
By choosing $k$ sufficiently large, this gives an embedding
\[\mathcal{C}_{r \cdot \Omega'_1, \tilde{\delta}_{r \cdot \Omega'_1}}\sqcup \mathcal{C}_{\Omega'_2,\delta_{\Omega'_2}} \to (\mathbb{C}P^2 \# (m+n) \overline{\mathbb{C}P^2},\omega)\]
for $r < 1$ arbitrarily close to $1$, and appropriate choice of $\tilde{\delta}_{r \cdot \Omega'_1}.$  Hence, by 
Proposition~\ref{prop:surgery}, there is a symplectic embedding
\[X_{r \cdot \Omega'_1} \to X_{\Omega_2'},\]
hence by Corollary~\ref{cor:nonrational} a symplectic embedding
\begin{equation}
\label{eqn:neededresult}
\op{int}(X_{\Omega'_1}) \to X_{\Omega'_2}.
\end{equation}

{\em Step 5.}  By combining \eqref{eqn:firstlambdaequation}, \eqref{eqn:secondlambdaequation}, and \eqref{eqn:neededresult}, we therefore have an embedding
\[ \frac{1}{\lambda^2} X_{\Omega_1} \to X_{\Omega_2},\]
for any $\lambda > 1$.  Hence, by corollary~\ref{cor:nonrational}, there exists an embedding
\[ \op{int}(X_{\Omega_1}) \to X_{\Omega_2},\]
which must necessarily have image in the interior of $X_{\Omega_2}$.  This completes the proof of the theorem.

\end{proof}

\begin{remark}
It is a very interesting problem to try to understand embeddings of other kinds of toric domains.  For example, one can ask under what conditions on $(a,b,c)$  there exists a symplectic embedding 
\begin{equation}
\label{eqn:polydiscball}
P(a,b) \to B^4(c).
\end{equation}  
It does not seem possible to answer this question using only the methods in this paper.  For studying embeddings as in \eqref{eqn:polydiscball} in a systematic way using something like the method we develop here, a first step would be to find an analogue of Proposition~\ref{prop:surgery} for embeddings with domain a polydisc, such that there exists a natural symplectic embedding of the resulting chain of spheres up to differences in symplectic areas.  The method in \S\ref{sec:blowup} can not be used to do this.  

It would be valuable to explore whether any scheme at all like what is done in this paper could be used to study \eqref{eqn:polydiscball}, or to study similar problems.  Certainly new ideas would be needed for this.  It is important to warn though that there are definitely differences between embeddings of concave toric domains into convex ones, and embedding problems like \eqref{eqn:polydiscball}.  For one thing, as has been already pointed out, ECH capacities do not always give a sharp obstruction to \eqref{eqn:polydiscball}.  Also, by Proposition~\ref{prop:connectivity}, symplectic embeddings of concave toric domains into convex ones are unique up to isotopy when they exist, which is known not to be true for certain problems like \eqref{eqn:polydiscball}.  There are natural symplectic packings of polydiscs by balls, and one could hope for some generalizations of the weight sequence and Theorem~\ref{thm:maintheorem} along these lines.  However, this looks problematic as well.  For example, the polydisc $P(1,2)$ has a natural decomposition into four disjoint $B^4(1)$, but it is not true that an embedding of $\sqcup_{i=1}^4 B^4(1)$ implies the existence of an embedding $P(1,2)$ since the former domain embeds into $B^4(2)$ but as stated in the introduction the domain $P(1,2)$ does not.

\end{remark}

\subsection{Examples}
\label{sec:examples} 

We now present several illustrative examples.  

\begin{example} {\em Weight sequences are not unique.}
Let $\Omega$ be the rectangle with vertices $(0,0), (1,0), (0,1)$ and $(1,1)$, and let $\Omega'$ be the triangle with vertices $(0,0), (2,0)$ and $(0,1)$.  Then $X_\Omega$ is a polydisk and $X_\Omega'$ is an ellipsoid.  Both $\Omega$ and $\Omega'$ are convex (we could also regard $\Omega'$ as concave, although for this example we do not want to), and the weight sequence for both is given by $(2,1,1);$ in particular, both have the same weight sequence.  This shows that weight sequences are not unique.  Also, by Theorem~\ref{thm:keytheorem}, a concave domain embeds into $X_{\Omega}$ if and only if it embeds into $X_{\Omega'}$.  This generalizes a result of Frenkel and Mueller \cite[Cor. 1.5]{fm}, which proves this when the domain is an ellipsoid (our proof is also different from theirs).      
\end{example}

\begin{example} {\em Constraints on weight sequences?}  
Let $(a_0,\ldots,a_n)$ be any finite sequence of nonincreasing real numbers.  We now explain why we can always construct a concave toric domain with weight sequence $(a_0,\ldots,a_n)$.  This concave domain will have the property that at each step in the inductive definition of the weight sequence, the domain $\Omega_2'$ from \S\ref{sec:weightsequences} is empty (we will call such a domain {\em short}).  By induction, we can assume that we can construct a short rational concave domain $\Omega_0$ with weight sequence $(a_1,\ldots,a_n).$  Now, consider the triangle $\Delta(a_0)$ with vertices $(0,0), (a_0,0)$ and $(0,a_0)$.  Multiply $\Omega_0$ by the matrix $\left( \begin{smallmatrix} 1&-1\\ 0& 1 \end{smallmatrix} \right)$ and then translate the result by $(a_0,0)$.  Let $\Omega$ be formed by taking the union of this region with $\Delta(a_0)$.  Then by construction $\Omega$ is a short concave domain with weight sequence $(a_0,\ldots,a_n)$.  Thus, any possible ball packing problem of a ball can arise by applying Theorem~\ref{thm:keytheorem}.  This is to be compared with the case of embedding an ellipsoid into a ball.  For example, it is shown in \cite[Lem. 1.2.6]{ms} that if $a=p/q$ is rational, then the weights $(a_1,\ldots,a_m)$ of $E(1,a)$ are required to satisfy
\[ \sum_{i=1}^m a_i^2 = a, \quad \quad \sum_{i=1}^m a_i  = a + 1 - \frac{1}{q}.\]
\end{example}

\begin{example}  {\em Billiards.}   Another simple example of a symplectic four-manifold is the {\em Lagrangian bidisk}
\[ P_L  \eqdef \lbrace (p_1,q_1,p_2,q_2) \in \mathbb{R}^4 \hspace{2 mm} | \hspace{2 mm} p_1^2 + p_2^2 \le 1, q_1^2 + q_2^2 \le 1 \rbrace.\]
The domain $P_L$ is the state space for a circular billiard table, and is of interest in dynamics.  After the first version of this paper appeared, Ramos \cite{ramos} showed that the interior of $P_L$ is in fact symplectomorphic to the interior of a concave toric domain.  Thus, Theorem~\ref{thm:maintheorem} can be used to produce embeddings of $\op{int}(P_L)$ into many targets.  Ramos used Theorem~\ref{thm:maintheorem} to produce optimal embeddings of $\op{int}(P_L)$ into balls and ellipsoids, for example he showed that there is a symplectic embedding
\[ P_L \to B^4(3 \sqrt{3})\]
and no embedding into a smaller ball exists, answering a question of Ostrover.  

Ramos' argument involves producing a toric action on subsets of $P_L$ and examining the moment image.  As mentioned in Example~\ref{ex:mainexample}, convex toric domains naturally arise from toric actions on closed symplectic manifolds, and it would be interesting to look for situations as in the case of $P_L$ where concave toric domains naturally arise from toric actions on noncompact sets.

\end{example}

\begin{example} {\em Flexibility.}  For convex toric domains $X_{\Omega}$, determining the set of $a$ such that there exists a symplectic embedding
\begin{equation}
\label{eqn:flexibilityequation}
\op{int}(E(1,a)) \to \sqrt{\frac{a}{\op{vol}(X_{\Omega})}} \cdot \op{int}(X_{\Omega})
\end{equation}
is often subtle.  Here, by $\op{vol}(X_{\Omega})$ we mean twice the area of $\Omega$; the equation \eqref{eqn:flexibilityequation} then implies that the ellipsoid fills all of the volume of the target, so we call such an embedding a {\em full filling}. For example, let $\mathcal{T}$ be the trapezoid with vertices $(0,0), (0,1), (1,1)$ and $(2,0)$; this is the moment polytope for the {\em first Hirzebruch surface}.  The weight sequence of $X_{\mathcal{T}}$ is $(2;1).$  By combining Theorem~\ref{thm:keytheorem} with the algorithm from \cite[\S 2.3]{b}, there exists a symplectic embedding
\[ E\left(1,3 \cdot (49/30)^2\right) \to 49/30 \cdot \op{int}(X_{\mathcal{T}}),\]
in particular for $a = 3 \cdot (49/30)^2 \approx 8.0033$ a full filling of $\op{int}(X_{\mathcal{T}})$ as in \eqref{eqn:flexibilityequation} exists.  

However, for $8 \le a < 3 \cdot (49/30)^2$, no such full filling exists.  We can see this as follows.  First, as we explain in the appendix, we can compute the ECH capacities of $X_{\mathcal{T}}$ using Theorem~\ref{thm:echcapacities}.  We can also compute the ECH capacities of $E(a,b)$ by the formula we review in \S\ref{sec:echreview}, see \eqref{eqn:ellipsoidexample}.  In particular, we have
\[ c_{175}(X_{\mathcal{T}}) = 30, \quad c_{175}(E(1,8))  = 49,\]
so by the Monotonicity Axiom for ECH capacities \eqref{eqn:monotonicityequation} and the Scaling Axiom \eqref{eqn:scaling}, if for $a = 8$ there exists a symplectic embedding
\[ \op{int}(E(1,a)) \to \lambda \cdot \op{int}(X_{\mathcal{T}})\]
then we must have $\lambda \ge 49/30.$  In particular, if $a \ge 8$, then the same constraint on $\lambda$ must hold, so for $8 \le a < 3 \cdot (49/30)^2$ we can not have a full filling, as claimed.

\end{example}

\begin{example}  {\em   A sample calculation.}  We now work through a more extended example in detail, see Figure~\ref{fig:example}.   

\label{ex:extendedexample}

\begin{figure}[t]
\centering
\includegraphics[scale=1]{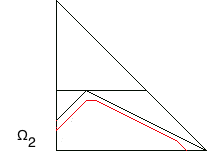}
\caption{The target for Example~\ref{ex:extendedexample}.  We have drawn the canonical decomposition given by the weight sequence (remember that the weight sequence for $\Omega_2$ gives a decomposition of the complement of $\Omega_2$ in a ball). The upper boundary of the inner approximation of $\Omega_2$ is also shown.}       
\label{fig:example}

\end{figure}

Let $\Omega_1$ be the domain whose upper boundary has vertices 
\[(0,10/3), (2/3,4/3),  (4/3,2/3), (7/3,0),\] 
and let $\Omega_2$ be the domain whose upper boundary has vertices 
\[(0,1), (1,2), (5,0).\]  
Then the weight expansion of $\Omega_1$ is $(2,2/3,2/3,1/3,1/3)$ and the weight expansion of $\Omega_2$ is $(5,3,2,1)$, see Figure~\ref{fig:example}.  

By Theorem~\ref{thm:keytheorem}, to see if $\op{int}(X_{\Omega_1})$ embeds into $\op{int}(X_{\Omega_2})$, it is equivalent to see if there is a ball packing
\begin{equation}
\label{eqn:theballpacking}
\op{int}\left(B(2/3) \sqcup B(2/3) \sqcup B(2) \sqcup B(1/3) \sqcup B(1/3) \sqcup B(3) \sqcup B(2) \sqcup B(1)\right) \to B(5).
\end{equation}

One can check, e.g. by applying the algorithm from \cite[\S 2.3]{b}, that in fact such a ball packing exists.  Hence, there is a symplectic embedding $\op{int}(X_{\Omega_1}) \to \op{int}(X_{\Omega_2})$.  In fact, this embedding is optimal (e.g. by \cite[\S 2.3]{b} again applied to \eqref{eqn:theballpacking}), in the sense that no larger scaling of $\op{int}(X_{\Omega_1})$ embeds into $\op{int}(X_{\Omega_2})$. 

To illustrate the concepts from the previous sections, note that there are five spheres in the chain of spheres corresponding to the blow up of $r \cdot \Omega_1$.  Each sphere corresponds to a blow up, and if we label these spheres in the order that they appear as edges of the outer approximation (with the first sphere the left most edge), and label the blow ups they correspond to accordingly, then the spheres, from left to right, have homology classes $E_1, E_2-E_1, E_3-E_2-E_4-E_5, E_4$ and $E_5-E_4$.  

There are four spheres in the chain of spheres corresponding to the blow up of $\Omega_2$ (including the sphere corresponding to the line at infinity).  If we label these spheres and the blowups with the same ordering convention as above, then they have homology classes $\widehat{E}_1, \widehat{E}_2-\widehat{E}_1-\widehat{E}_3, \widehat{E}_3$, and $L-\widehat{E}_2-\widehat{E}_3$.  

The cohomology class of the symplectic form on the blow up is given in this notation by
\begin{align}
[\omega_1]  & =  5L - (2/3)re_1-(2/3)re_2-2re_3 \nonumber \\
 & \qquad  -(1/3)re_4-(1/3)re_5 \nonumber \\
 & \qquad -\widehat{e}_1-3\widehat{e}_2-2\widehat{e}_3\nonumber \\
 & \qquad - \sum_{i=1}^5 \op{err}_i(\delta_1) e_i - \sum_{j=1}^3 \op{err}_j(\delta_2)\widehat{e}_j .
\end{align}

\end{example}

\section{ECH capacities give sharp obstructions to embeddings of concave domains into convex ones} 
\label{sec:sharpness}

In this section we prove Theorem~\ref{thm:maintheorem}.  We will first review the definition of ECH capacities in the cases that we need.  We will also review some formal properties that will be used in the proof.    

\subsection{Embedded contact homology}
\label{sec:echreview}

Let $Y$ be a closed oriented three-manifold.  A {\em contact form} on $Y$ is a one-form $\lambda$ satisfying
\[ \lambda \wedge d \lambda > 0.\] 
A contact form determines a canonical vector field $R$ by the equations
\[ \lambda(R) = 1, \quad d \lambda(R,\cdot) = 0.\]
The vector field $R$ is called the {\em Reeb vector field}, and the closed orbits of $R$, called {\em Reeb orbits}, are of considerable interest.  A contact form $\lambda$ is called {\em nondegenerate} if all Reeb orbits for $\lambda$ are cut out transversally, see \cite[\S 1.3]{hutchingslecture} for the precise definition. 

Let $(Y,\lambda)$ be a closed three-manifold with nondegenerate contact form.  The {\em embedded contact homology} of the pair $(Y,\lambda)$, denoted $ECH_*(Y,\lambda),$ is the homology of a chain complex $ECC_*(Y,\lambda).$  The chain complex $ECC_*(Y,\lambda)$ is freely generated over $\mathbb{Z}/2\mathbb{Z}$ by finite {\em orbit sets}
\[ \alpha = \lbrace (\gamma_i,m_i ) \rbrace,\]
where the $\gamma_i$ are distinct embedded Reeb orbits and the $m_i$ are positive integers, with the constraint that $m_i=1$ whenever $\gamma_i$ is hyperbolic.  (We could also define the chain complex over $\mathbb{Z}$, but for the applications in this paper we do not need this.)   The chain complex differential $d$ counts ``ECH index $1$" $J$-holomorphic curves in $\mathbb{R} \times Y$, for a generic compatible almost complex structure $J$.  The ECH index induces a grading $*$ on $ECC$ such that the differential decreases the grading by $1$.  Taubes has shown that there is a canonical isomorphism
\[ ECH_*(Y,\lambda) \cong \widehat{HM}^{-*}(Y),\]
where $\widehat{HM}$ denotes the {\em Seiberg-Witten Floer cohomology} defined by Kronheimer-Mrowka in \cite{km}.  In particular, the homology $ECH(Y,\lambda)$ depends neither on the choice of almost complex structure $J$, nor on $\lambda$, and so we sometimes denote it $ECH(Y)$.  For more details about the above, see \cite{hutchingslecture}.

Reeb orbits $\gamma$ have an action $\mathcal{A}(\gamma) = \int_\gamma \lambda$ which we can extend to a filtration on $ECH$.  Specifically, if $\alpha = \lbrace (\gamma_i,m_i) \rbrace $ is an orbit set, define the {\em action} of $\alpha$
\[ \mathcal{A}(\alpha) = \sum_i m_i \int_{\gamma} \lambda ,\]
and let $ECC^L(Y,\lambda)$ denote the subspace generated by orbit sets $\alpha$ with $\mathcal{A}(\alpha) < L.$  The differential restricts to $ECC^L$, so the homology $ECH^L(Y,\lambda)$ is well-defined and there is an inclusion induced map $ECH^L(Y,\lambda) \to ECH(Y)$.  If $\sigma$ is a nonzero class in $ECH$, we can define the action required to represent it by
\[ c_{\sigma}(\lambda) \eqdef \op{inf} \lbrace L \hspace{1mm} | \hspace{1 mm} \sigma \in \op{Im}\left(ECH^L(Y,\lambda) \to ECH(Y) \right) \hspace{1 mm} \rbrace. \] 
The number $c_{\sigma}(\lambda)$ is called the {\em spectral invariant} associated to $\sigma$.  When $\lambda$ is degenerate, we can still define $c_{\sigma}(\lambda)$ by taking the limit of $c_{\sigma}(\lambda_n)$ as $\lambda_n \to \lambda$ in $C^0$, see \cite{qech}. 

\subsection{ECH capacities}

We would like to use $ECH$ to define symplectic capacities.  This is most natural when $(X,\omega)$ is a symplectic $4$-manifold with boundary, such that $\omega=d \lambda$ and $\lambda |_{\partial X}$ is a contact form.  When $\partial X$ is oriented positively with respect to $\omega^2$, we call such an $(X,\omega)$ a {\em Liouville domain}.  For example, any concave or convex toric domain is a Liouville domain.    

In our case, where $X$ is a concave or convex toric domain, $\partial X = S^3$.  The embedded contact homology of $S^3$ has a canonical $\mathbb{Z}$ grading in this case, and $ECH_*(S^3)$ is given by 
\[ ECH_{2k}(S^3)_{2k} = \mathbb{Z}/2\mathbb{Z}, k \ge 0, \quad \quad ECH_*(S^3) = 0, \hspace{1 mm} \op{otherwise},\]
as explained in \cite{hutchingslecture}.  In particular, for each nonnegative integer $k$, there are canonical nonzero classes $\sigma_k$ in grading $2k$.  If $\omega$ is concave or convex, we now define the $k^{th}$ ECH capacity
\[ c_k(X,\omega) \eqdef c_{\sigma_k}(\lambda),\]
where $\lambda$ is the restriction of the standard one-form $\lambda_{std} = \frac{1}{2} \sum_i (x_i  dy_i - y_i dx_i)$ on $\mathbb{R}^4$ to $\partial X$.  One can modify this definition to define ECH capacities for any Liouville domain, and in fact ECH capacities can be defined for any symplectic $4$-manifold, see \cite{qech}.  

\begin{example}
\label{ex:ellipsoidexample}

The ECH capacities of the ellipsoid were computed in \cite{qech}.  The $k^{th}$ ECH capacity of the ellipsoid $E(a,b)$ is the $(k+1)^{st}$ smallest element in the matrix
\begin{equation}
\label{eqn:ellipsoidexample}
(ma+nb)_{ (m,n) \in \mathbb{Z}_{\ge 0} \times \mathbb{Z}_{\ge 0}}.
\end{equation}
For example, the ECH capacities of the ball $E(1,1)$ start with 
\[ 0, 1, 1, 2, 2, 2, 3, 3, 3, 3, \ldots .\]
We will give a formula for the ECH capacities of convex domains in terms of ECH capacities of balls in Theorem~\ref{thm:echcapacities}.
 
\end{example}

As mentioned in the introduction, the ECH capacities satisfy the key Monotonicity Property \eqref{eqn:monotonicityequation}.  To state another property that they satisfy, recall the {\em sequence summation} operation from \cite{qech}, defined for sequences $S$ and $T$ indexed starting at $k=0$ by
\[ (S \# T)_k = \op{sup}_{i + j = k} (S_i + T_j).\]
Here, the notation $A_i$ denotes the $i^{th}$ term of the sequence $A$.  We can now state the Disjoint Union axiom \cite{qech} for the sequence of ECH capacities $c_{ECH}$, which says that for Liouville domains $X_1, X_2$,
\begin{equation}
\label{eqn:disjointunion}
c_{ECH}(X_{\Omega_1} \coprod X_{\Omega_2} ) = c_{ECH}(X_{\Omega_1}) \# c_{ECH}(X_{\Omega_2}).
\end{equation}
Another useful axiom is the Scaling Axiom, also proved in \cite{qech}, which says that
\begin{equation}
\label{eqn:scaling}
 c_k (X, r \cdot \omega) = r \cdot c_k (X,\omega),
 \end{equation}
for any positive real number $r.$

\subsection{Sharpness for the ball packing problem implies sharpness for ECH capacities}

We now explain the proof of Theorem~\ref{thm:maintheorem}.  The key point is that it was shown in \cite{qech} that ECH capacities are known to give sharp obstructions to symplectic ball packing problems.

\begin{proof}[Proof of Theorem~\ref{thm:maintheorem}] 

Let $\Omega_1$ be concave and $\Omega_2$ convex.  We need to show that  $\op{int}(X_{\Omega_1})$ embeds into $\op{int}(X_{\Omega_2})$ if and only if $c_k(\op{int}(X_{\Omega_1})) \le c_k(\op{int}(X_{\Omega_2}))$ for all $k$.  The fact that a symplectic embedding
\[\op{int}(X_{\Omega_1}) \to \op{int}(X_{\Omega_2})\]
implies that 
\begin{equation}
\label{eqn:capacityassumption}
c_k(\op{int}(X_{\Omega_1})) \le c_k(\op{int}(X_{\Omega_2}))
\end{equation} 
for all $k$ follows from the Monotonicity property \eqref{eqn:monotonicityequation}.

{\em Step 1.} We first prove the converse assuming that $\Omega_1$ and $\Omega_2$ are rational.

By the Monotonicity Axiom \eqref{eqn:monotonicityequation}, the Disjoint Union property \eqref{eqn:disjointunion},  and the proof of the ``only if" direction of Theorem~\ref{thm:keytheorem}, we know that
\[(c_{ECH}(\op{int}(X_{\Omega_2})) \# c_{ECH}( \op{int}(\widehat{B}(\Omega_2))))_k\le c_{k}(B(b))\]
for all $k$.
We also know that for any $k$,  
\[ c_k(\op{int}(B(\Omega_1))) \le c_k(\op{int}(X_{\Omega_1})).\]
Since sequence sum against a fixed sequence respects inequalities, we can combine \eqref{eqn:capacityassumption} with the above inequalities to find that
\begin{equation}
\label{eqn:inequalitywewant}
c_k(\op{int}(B(\Omega_1)) \sqcup \op{int}(\widehat{B}(\Omega_2))) \le c_k(B(b))
\end{equation}
for all $k$.  It is known that ECH capacities give sharp obstructions to all (open) ball packings of a ball, see e.g. \cite{qech}.  Hence, \eqref{eqn:inequalitywewant} implies that there exists a symplectic embedding 
\[ \op{int}(B(\Omega_1)) \sqcup \op{int}(\widehat{B}(\Omega_2)) \to B(b).\]
Hence by Theorem~\ref{thm:keytheorem}, there exists a symplectic embedding
\[ \op{int}(X_{\Omega_1}) \to \op{int}(X_{\Omega_2}),\]
hence the theorem in the rational case.

{\em Step 2.}  We now deduce Theorem~\ref{thm:maintheorem} in general by using the result from the previous step.

Given $\Omega_1$ concave and $\Omega_2$ convex, for each $\lambda > 1$ we can find a rational concave set $\Omega'_1$ and a rational convex set $\Omega'_2$ such that 
\[ \frac{1}{\lambda} \Omega_1 \subset \op{int}(\Omega_1') \subset \Omega_1 ,\]
and
\[ \Omega_2 \subset \op{int}(\Omega'_2) \subset \lambda \Omega_2.\]
By combining the above inclusions with \eqref{eqn:monotonicityequation} and \eqref{eqn:capacityassumption}, it follows from the previous step that there is a symplectic embedding
\[ \op{int}(X_{\Omega'_1}) \to \op{int}(X_{\Omega'_2}).\]
Hence, by again applying the above inclusions, there is a symplectic embedding
\[ \frac{1}{\lambda^2} \cdot X_{\Omega_1} \to X_{\Omega_2}.\]
By letting $\lambda$ tend to $1$ and applying Corollary~\ref{cor:nonrational}, we therefore get a symplectic embedding
\[ \op{int}(X_{\Omega_1}) \to X_{\Omega_2},\]
which must necessarily have image in the interior of $X_{\Omega_2}.$

\end{proof}

\appendix

\section{Appendix (by Keon Choi and Daniel Cristofaro-Gardiner): The geometric meaning of ECH capacities of convex domains}
\label{sec:appendix}

\subsection{The main theorem} 

We assume below that the reader is familiar with the definitions and notation from the body of this paper.  There, the second author showed that ECH capacities give a sharp obstruction to embedding any concave toric domain into a convex one.  The basic idea of the proof was to show that a concave domain embeds into a convex one if and only if it is possible to symplectically embed a certain collection of balls into another ball.  This suggests that there should be a close relationship between the ECH capacities of concave or convex toric domains, and the ECH capacities of balls.  

In \cite{concave}, the authors and Frenkel, Hutchings and Ramos showed that ECH capacities of any concave toric domain are given by the ECH capacities of the disjoint union of the balls determined by the weight sequence of the domain, see Theorem~\ref{thm:concavetheorem} for the precise statement.  The purpose of this appendix is to prove a similar formula for convex domains.

To state our formula, we recall the {\em sequence subtraction} operation that is implicit in \cite{qech} and was first explicitly defined in \cite{blog}.  This is given for nondecreasing sequences $S$ and $T$, indexed starting at $0$ and with $T \le S$, by
\begin{equation}
\label{eqn:definitioninf}
(S - T)_k \eqdef \op{inf}_{l \ge 0} S_{k+l}-T_l.
\end{equation}
(Here, the notation $T \le S$ means that $T_i \le S_i$ for every index $i$.)  The operation $\#$ and $-$ are related by the inequalities
\begin{equation}
\label{eqn:sequenceoperations}
(S-T) \# T \le S \le (S\#T) -T. 
\end{equation}

For our purposes, the sequence subtraction operation is significant because of the following:

\begin{theorem}
\label{thm:echcapacities} 
Let $X_{\Omega}$ be a convex toric domain, let $b$ be the head of the weight expansion for $\Omega$, and let $b_i$ be the $i^{th}$ term in the negative weight expansion for $\Omega$.  
Then 
\begin{equation}
\label{eqn:geometricequation}
c_{ECH} (X_{\Omega}) = c_{ECH}(B(b)) - c_{ECH}(\coprod_i B(b_i)).
\end{equation} 
\end{theorem}

Note that it follows from the Monotonicity and Scaling axioms that $c_k(X_{\Omega})=c_k(\op{int}(X_{\Omega}))$ for any convex toric domain $X_{\Omega}$.  Note also that even when $\Omega$ is not rational, the above formula still makes sense, see \cite[Rmk. 1.6]{concave}.  For the formula for the ECH capacities of a ball, see Example~\ref{ex:ellipsoidexample}.

\begin{remark}
\label{rmk:whymin}

If $T \le S$, and  
\[ \op{lim}_{i \to \infty} S_i - T_i = + \infty,\]
write $T < S.$  If $T < S$, then we are justified in replacing the infimum in \eqref{eqn:definitioninf} with a minimum.  When $X_1$ and $X_2$ are Liouville domains with all ECH capacities finite,  and $T=c_{ECH}(X_1)$ and $S=c_{ECH}(X_2)$ are sequences of ECH capacities, we have $T < S$ whenever $\op{vol}(X_1) < \op{vol}(X_2)$.  This follows from \cite[Thm. 1.1]{vc}.  
  
\end{remark}

We can regard Theorem~\ref{thm:echcapacities} as expressing a fundamental limitation of the strength of ECH capacities of convex domains.  For example, we have:

\begin{corollary}
\label{cor:limitation}
Let $\Omega$ be convex and let $X$ be any Liouville domain with all ECH capacities finite.   Let $b$ be the head of the weight expansion for $\Omega$, and let $b_i$ be the $i^{th}$ term in the negative weight expansion for $\Omega$.  If we have
\begin{equation}
\label{eqn:assumedembedding} 
c_k (X \sqcup (\coprod_i B(b_i))) \le c_k(B(b))
\end{equation}
for all $k$, then we must have 
\begin{equation}
\label{eqn:desiredequation}
c_k(X) \le c_k(X_{\Omega})
\end{equation}
for all $k$.  

\end{corollary}

\begin{proof}

By combining the Disjoint Union axiom \eqref{eqn:disjointunion} and \eqref{eqn:assumedembedding}, we have
\[ c_{ECH}(X) \# c_{ECH}( \coprod_i B(b_i) ) \le c_{ECH}(B(b)).\]
Now subtract $c_{ECH}(\coprod_i B(b_i))$ from both sides of this equation and apply \eqref{eqn:sequenceoperations} to get 
\[c_{ECH}(X) \le c_{ECH}(B(b)) - c_{ECH}(\coprod_i B(b_i)).\]  
Now apply Theorem~\ref{thm:echcapacities} to get \eqref{eqn:desiredequation}.

\end{proof}

\begin{remark}

The analogue of Theorem~\ref{thm:echcapacities} was proved in the concave case in \cite{concave}.  There, the authors show:

\begin{theorem}\cite[Thm. 1.4]{concave}

\label{thm:concavetheorem}

Let $\Omega$ be concave, and let $a_i$ be the $i^{th}$ weight of $\Omega$.  Then
\begin{equation}
\label{eqn:concaveformula}
c_k(X_{\Omega}) = c_k ( \coprod_i B(a_i) ).
\end{equation}

\end{theorem}

The equation \eqref{eqn:concaveformula} will be used in the proof of Theorem~\ref{thm:echcapacities}.

\end{remark}

\subsection{Lattice points and $\Omega$-lengths}

The ECH capacities of concave and convex domains are related to certain lattice point counts.  We now introduce the terms we need to make this precise.

We first define the upper boundaries of the regions we need to consider.

\begin{definition}
\label{defn:lambdalength}

Let $\Lambda:[0,c] \to \mathbb{R}^2$ for some $c \ge 0$ be a polygonal path in the plane, with vertices at lattice points.  Assume that the tangent $\Lambda'$ is nonzero on $[0,c] \setminus \{0 = c_0 < \cdots < c_n = c\}$, where the $\Lambda(c_i)$ are the vertices of $\Lambda$.  In addition, for any nonzero vector $v \in \mathbb{R}^2$, let $\theta(v)$ be the number $\theta \in [0,2\pi)$ so that $v$ is a positive multiple of $(\sin \theta, \cos \theta)$.
\begin{itemize}
\item An {\em edge} of $\Lambda$ is the displacement vector between consecutive vertices of $\Lambda$.  

\item $\Lambda$ is a {\em lattice path} if its vertices are lattice points and  $\Lambda(0)=(0,y(\Lambda))$ and $\Lambda(c) = (x(\Lambda),0)$ for nonnegative integers $x(\Lambda)$ and $y(\Lambda)$.
\item $\Lambda$ is {\em concave} if $\theta(\Lambda')$ is nonincreasing and takes values in $(\pi/2,\pi)$.
\item $\Lambda$ is {\em convex} if $\theta(\Lambda')$ is nondecreasing and takes values in $(0,3\pi/2)$.
\end{itemize}
\end{definition}

The paths $\Lambda$ have an $\Omega$-length, defined by the domain $\Omega$, which will also be important.

\begin{definition}
Let $X_\Omega$ be a convex toric domain and $\Lambda$ a convex lattice path. If $\nu$ is any vector in $\mathbb{R}^2$, let $p_{\Omega,\nu}$ be a point on the boundary of $\Omega$ such that $\Omega$ lies entirely in the ``right half-plane'' of the line through $p_{\Omega,\nu}$ in the direction $\nu$.  More precisely, for any $p \in \Omega$, we have 
\begin{equation}
\label{eqn:tangentequation}
(p-p_{\Omega,\nu}) \times \nu \ge 0
\end{equation} 
where $\times$ denotes the cross product. Define
\begin{equation}
\label{eqn:defnv} 
\ell_{\Omega}(\nu) = \nu \times p_{\Omega,v},
\end{equation}
and if $\Lambda$ is a convex lattice path, define
\begin{equation}
\label{eqn:definell}
 \ell_{\Omega}(\Lambda) = \sum_{\nu \in \op{Edges}(\Lambda)} \ell_{\Omega}(\vu).
\end{equation}
If $X_\Omega$ is a concave toric domain and $\Lambda$ is a concave lattice path, $\ell_{\Omega}(\Lambda)$ is defined by \eqref{eqn:defnv} and \eqref{eqn:definell}, where $p_{\Omega,\nu}$ is a point on the boundary of $\Omega^c:=[0,\infty)^2 \setminus \Omega$ so that $\Omega^c$ lies entirely on the ``left half-plane'' of the line through $p_{\Omega,\nu}$ in the direction $\nu$.
\end{definition}

We will also want to count lattice points in regions bounded by $\Lambda$.  We now make this precise.  

\begin{definition}\label{def:latticecount}
If $\Lambda$ is a convex lattice path, let $\widehat{\mathcal{L}}_{\Omega}(\Lambda)$ denote the count of lattice points in the region enclosed by $\Lambda$ and the axes, including all the lattice points on the boundary. If $\Lambda$ is a concave lattice path, let $\widecheck{\mathcal{L}}(\Lambda)$ denote the number of lattice points in the region enclosed by $\Lambda$ and the axes, not including lattice points on $\Lambda$ itself.
\end{definition}

\begin{example}

Using this terminology, we can state an alternative formula for the ECH capacities of concave toric domains.  Namely, we have:

\begin{theorem}\cite[Thm. 1.21]{concave}  

Let $\Omega$ be concave.  Then
\begin{equation}
\label{eqn:concaveformulalength}
c_k(X_{\Omega}) = \op{max} \lbrace \ell_{\Omega} (\Lambda) | \widecheck{\mathcal{L}}(\Lambda) = k \rbrace,
\end{equation}
where the maximum runs over all concave lattice paths.

\end{theorem}

We will also use this fact in the proof.

\end{example}

\begin{remark}
\label{rmk:itsanorm}

When $\Omega$ is convex, $\ell_{\Omega}$ is in fact a (non-symmetric) norm: it satisfies the scaling axiom $\ell_{\Omega}(c \cdot v) = c \cdot \ell_{\Omega}(v)$ whenever $c \ge 0$, and it satisfies the triangle inequality
\begin{equation}
\label{eqn:triangleinequality}
\ell_{\Omega}(v+w) \le \ell_{\Omega}(v) + \ell_{\Omega}(w).
\end{equation}
To see why \eqref{eqn:triangleinequality} holds, for a fixed $v \in \mathbb{R}^2$ consider the function
\[ x \to v \cdot x\]
on $\mathbb{R}^2$.  This is maximized over $\Omega$ on $\partial \Omega$, at points $x$ at which $v$ is normal to $\Omega$ and pointing outward.  (When $\partial \Omega$ is not smooth, we consider any vector $\nu$ such that \eqref{eqn:tangentequation} holds a tangent vector, and we consider any vector normal to a tangent vector a normal vector.)  It follows that $x \to \nu \times x$ is maximized at the point $p_{\Omega,\nu}$ from Definition~\ref{defn:lambdalength}.  We therefore have
\[ \ell_{\Omega}(v+w) = (v+w) \times p_{\Omega,v+w} \le v \times p_{\Omega,v} + w \times p_{\Omega,w},\]
hence \eqref{eqn:triangleinequality}.                 

\end{remark}

\subsection{ECH capacities of convex domains}
 
We can now give the proof of the main theorem of this appendix.

\begin{proof}[Proof of Theorem~\ref{thm:echcapacities}.]
Recall from \S\ref{sec:weightsequences} that the first step of the weight expansion for $X_{\Omega}$ determines regions $\Omega_1,\Omega_2$ and $\Omega_3$ such that $X_{\Omega_1}$ is a $B(b)$ and $X_{\Omega_2}$ and $X_{\Omega_3}$ are concave toric domains. For a given $k \ge 0$, we claim a series of inequalities
\begin{align}\label{eqn:ineq}
c_k(X_\Omega)  &\le \min_{k_1 - l = k} \{c_{k_1}(X_{\Omega_1}) - c_{l}( \coprod_i B(b_i))\} \nonumber \\
&= \min_{k_1-k_2-k_3 = k} \{ c_{k_1} (X_{\Omega_1}) - c_{k_2}(X_{\Omega_2}) - c_{k_3}(X_{\Omega_3})\} \nonumber \\ 
&\le \min\{\ell_\Omega(\Lambda) \hspace{1 mm} | \hspace{1 mm} \widehat{\mathcal{L}}(\Lambda) \ge k+1\} \\
&\le c_k(X_\Omega), \nonumber
\end{align}
which proves the theorem.  (We are justified in writing a minimum rather than an infimum throughout, by Remark~\ref{rmk:whymin}.)   Here and throughout the proof, $k_1,k_2,k_3$ and $l$ denote nonnegative integers.  We now explain the proofs of the above inequalities.

{\em Step 1.}   By the definition of the weight expansion, there is a symplectic embedding
\[ X_{\Omega} \sqcup (\coprod_i \op{int} (B(b_i)) \to B(b).\]
It then follows from the Monotonicity axiom \eqref{eqn:monotonicityequation} and the Disjoint Union property \eqref{eqn:disjointunion} that for any $k_1$ and $l$
\[ c_{k_1}(X_{\Omega}) + c_l(\coprod_i B(b_i)) \le c_{k_1 + l}(B(b)).\]
This proves the first inequality of \eqref{eqn:ineq}.

{\em Step 2.} Since the weights of $\Omega_2$ and $\Omega_3$ collectively correspond to the negative weights of $\Omega$, we have
\[ \max_{k_2+k_3 = l}\{c_{k_2}(X_{\Omega_2}) + c_{k_3}(X_{\Omega_3})\} = \max_{\sum l_i=l}\sum c_{l_i}(B(b_i)) = c_l(\coprod_i B(b_i)) \]
by \eqref{eqn:concaveformula}.  This proves the equality on the second line of \eqref{eqn:ineq}.

{\em Step 3.} To prove the third inequality of \eqref{eqn:ineq}, given any convex lattice path $\Lambda$ with 
$\widehat{\mathcal{L}}(\Lambda) - 1 = k_0 \ge k$, we show how to choose $k_1,k_2$ and $k_3$ with $k_1 - k_2-k_3 = k$ so that 
\begin{equation}\label{eqn:step2}
 \ell_\Omega(\Lambda) \ge c_{k_1}(X_{\Omega_1}) - c_{k_2}(X_{\Omega_2}) - c_{k_3}(X_{\Omega_3}).
\end{equation}
Write $\Lambda$ as a concatenation $\tilde\Lambda_2\tilde\Lambda_1\tilde\Lambda_3$ of paths so that $\theta(\tilde\Lambda_2')$, $\theta(\tilde\Lambda_1')$ and $\theta(\tilde\Lambda_3')$ take values in $(0,3\pi/4)$, $\{3\pi/4\}$ and $(3\pi/4,3\pi/2)$, respectively. As in the definition of the weight expansion, $\tilde\Lambda_2$ and $\tilde\Lambda_3$ are affine equivalent to concave lattice paths $\Lambda_2$ and $\Lambda_3$, respectively. Also, let $\Lambda_1$ denote the linear path from $(0,a)$ to $(a,0)$ extending $\tilde\Lambda_1$. We take $k_2=\mathcal{\widecheck L}(\Lambda_2)$, $k_3=\mathcal{\widecheck L}(\Lambda_3)$ and $k_1 = k + k_2 + k_3$. Observe that $\mathcal{\widehat L}(\Lambda_1) -1 = k_0 +k_2 + k_3 \ge k_1$. 

By \eqref{eqn:ellipsoidexample} and the fact that the ECH capacities of any symplectic manifold are nondecreasing, we then have
\[ \ell_{\Omega_1}(\Lambda_1)=c_{k_0+k_2+k_3}(B(b)) \ge c_{k_1}(B(b)). \]
By \eqref{eqn:concaveformulalength},
\[ \ell_{\Omega_2}(\Lambda_2) \le c_{k_2}(X_{\Omega_2}) \]
and
\[ \ell_{\Omega_3}(\Lambda_3) \le c_{k_3}(X_{\Omega_3}). \]
Moreover, by the argument in Step 4 of \cite[\S 2.1]{concave},
\[ \ell_{\Omega}(\Lambda)= \ell_{\Omega_1}(\Lambda_1) - \ell_{\Omega_2}(\Lambda_2)-\ell_{\Omega_3}(\Lambda_3). \]
We substitute the previously obtained bounds into the above to obtain \eqref{eqn:step2}.

{\em Step 4.}  Consider a dilation $\tilde\Omega$ of $\Omega$ by a factor $\lambda <1$ about an interior point of $\Omega$. Then, $X_{\tilde\Omega}$ is a disk bundle over $T^2$, and by \cite[Thm. 1.11]{qech}, there is a closed convex path $\tilde\Lambda$ with corners on lattice points so that $\mathcal{\widehat L}(\tilde\Lambda) = k+1$ and $\ell_{\tilde\Omega}(\tilde\Lambda) = c_k(\tilde\Omega)$. Here, $\mathcal{\widehat L}(\tilde\Lambda)$ denotes the number of lattice points in the region enclosed by $\tilde\Lambda$, including the ones on the boundary, and $\ell_{\tilde\Omega}(\tilde\Lambda)$ is defined by \eqref{eqn:definell} as in the case of a convex domain.

Consider the part $\Lambda$ of the path $\tilde\Lambda$ consisting only of edges with $0 < \theta(\nu) < 3\pi/2$. Then $\Lambda$ is a convex lattice path (after translation if necessary) with $\mathcal{\widehat L}(\Lambda) \ge k+1$ and $\ell_{\tilde\Omega}(\tilde\Lambda) = \lambda \ell_\Omega(\Lambda)$. Hence, by the Monotonicity axiom,
\[ \ell_\Omega(\Lambda) = c_k(\tilde\Omega)/\lambda \le c_k(\Omega)/ \lambda \]
Taking the limit as $\lambda \to 1$ proves the last inequality. 
\end{proof}

We close with the following analogue of the formula from \cite[Thm. 1.11]{qech}.
\begin{corollary}
\label{cor:convexcorollary}
Let $\Omega$ be a convex toric domain.  Then 
\[ c_k(X_{\Omega}) = \op{min} \lbrace \ell_{\Omega}(\Lambda) | \mathcal{\widehat L}(\Lambda) = k+1 \rbrace,\]
where the minimum is over convex lattice paths $\Lambda$.
\end{corollary}  

\begin{proof}
As part of the proof of Theorem~\ref{thm:echcapacities}, we saw that this formula holds when the minimum is taken over $\Lambda$ with $\mathcal{\widehat L}(\Lambda) \ge k+1$.  Thus, to prove the theorem, it suffices to show that given any lattice path $\Lambda$, there is another lattice path $\Lambda'$ with 
\begin{equation}
\label{eqn:neededequations}
\mathcal{\widehat L}(\Lambda') = \mathcal{\widehat L}(\Lambda) - 1, \quad \quad \ell_{\Omega}(\Lambda') \le \ell_{\Omega}(\Lambda).
\end{equation}
We can define such a path by using an analogue of the ``corner rounding" operation from 
\cite{hutchings}: specifically, given $\Lambda$ we define $\Lambda'$ by choosing any vertex other than the origin, removing the lattice point corresponding to that edge, and taking the convex hull of the remaining lattice points.  This satisfies the first equation in \eqref{eqn:neededequations} by definition, and it satisfies the second equality by the triangle inequality \eqref{eqn:triangleinequality}.

\end{proof}

\end{document}